\documentclass[12pt,reqno]{amsart}

\usepackage{amssymb,array}
\usepackage{amsmath}
\usepackage{amsthm}
\usepackage{amsbsy,mathrsfs}
\usepackage{bm}
\usepackage[english]{babel}
\usepackage{graphicx}
\usepackage{color}
\usepackage{graphicx}
\usepackage{color}
\usepackage{hyperref}
\usepackage[toc,page]{appendix}

\usepackage{thmtools}
\usepackage{thm-restate}

\usepackage{cleveref}

\divide
\oddsidemargin by 2 
\theoremstyle{plain}
\newtheorem{theorem}{Theorem}[section]
\newtheorem*{theorem*}{Theorem}
\newtheorem*{conj*}{Conjecture}
\newtheorem{lemma}[theorem]{Lemma}
\newtheorem{ex}[theorem]{Example}

\newtheorem{proposition}[theorem]{Proposition}
\newtheorem{computational theorem}[theorem]{Computational Theorem}

\theoremstyle{definition}

\newtheorem{remark}[theorem]{Remark}

\newtheoremstyle{named}{}{}{\itshape}{}{\bfseries}{.}{.5em}{\thmnote{#3 }#1}
\theoremstyle{named}
\newtheorem*{namedtheorem}{Theorem}

\newcommand{\C}{\mathbb{C}}
\newcommand{\K}{\mathbb{K}}

\newcommand{\Z}{\mathbb{Z}}
\newcommand{\N}{\mathbb{N}}

\newcommand{\e}{\varepsilon}

\renewcommand{\H}{{\mathbb{H}}}

\newcommand{\la}{\langle}
\newcommand{\ra}{\rangle}
\newcommand{\U}{\mathcal{U}}
\newcommand{\s}{\mathcal{S}}
\newcommand{\M}{\mathcal{M}}
\newcommand{\T}{\mathcal{T}}

\newcommand{\ox}{\alpha}
\newcommand{\oy}{\beta}

\newcommand\conju[1]{\widetilde{#1}}

\renewcommand{\a}{x}
\renewcommand{\b}{y}
\newcommand{\xy}{(x,y)\mbox{-intersection points}}
\newcommand{\xyp}{(x',y')\mbox{-intersection points}}

\theoremstyle{plain} 
\newcommand{\thistheoremname}{}
\newtheorem*{genericthm*}{\thistheoremname}

\newcommand\map[3]{\mbox{${#1}\colon {#2} \longrightarrow {#3}$}}
  
\newcommand\smap[2]{\mbox{${#1}\colon {#2} \longrightarrow {#2}$}}

\begin{document}
\title[The Lie bracket of   undirected  curves on a surface]{The   Lie  bracket of undirected closed curves on a surface }

\author{Moira Chas}
\address{Stony Brook Mathematics Department and Institute for Mathematical Sciences}
\email{moira.chas@stonybrook.edu}

\author{ Arpan Kabiraj}
\address{Department of Mathematics, Indian Institute of Technology Palakkad}
\email{arpaninto@iitpkd.ac.in}

\begin{abstract}
A Lie bracket defined on the linear span of the free homotopy  classes of undirected closed curves   was discovered in stages passing through  Thurston's  earthquake deformations, Wolpert's  corresponding calculations with Hamiltonian vector fields and Goldman's algebraic treatment of the  latter leading to a Lie  bracket on the span of directed closed curves.  The purpose of this work is to deepen the understanding of the former Lie bracket which will be referred to as the Thurston-Wolpert-Goldman Lie bracket or, briefly, the TWG bracket.  

We give a local direct geometric  definition of the TWG bracket and  use this geometric point of view to prove three  results: 
firstly, the center of the TWG-bracket is the Lie  sub algebra generated by the class of the trivial loop and the  classes of loops parallel to  boundary components or punctures; 
secondly the analogous result hold for the centers of the universal enveloping algebra and of the symmetric algebra determined by the TWG  Lie  algebra; 
and thirdly, 
 in terms of the natural basis, the TWG bracket of two non-central  curves   is always a linear combination of non-central curves.
We also  give a brief and more illuminating proof of a known result, namely, the TWG bracket counts intersection. 

We conclude by discussing  substantial  computer evidence suggesting an unexpected and strong  conjectural statement relating the intersection structure of curves and the TWG bracket, namely, if the TWG bracket of two distinct undirected  curves is zero then these curve classes have disjoint representatives.

The main tools are basic hyperbolic geometry and Thurston's earthquake theory.

\end{abstract}

\maketitle


\section{Introduction} 

\subsection{ Geometric definition of the Thurston-Wolpert-Goldman Lie algebra}

Let $\Sigma$ be an oriented Riemann surface, not necessarily of finite type, which carries a complete metric of constant curvature  $-1$.

Given two undirected  curves $x$ and $y$  on $\Sigma$ and a transversal  intersection point  $P$ of $x$ and $y$, one can define a new   curve  ${(x*_P y)}_0$ as follows: orient $x$ and $y$ in such a way that the orientation of the surface  agrees with the orientation determined by the two  ordered oriented branches  of $x$ and $y$ emanating from $P$,  then  cut the two branches at $P$ and reconnect the curves following the new orientations (that is, perform the usual loop product at $P$), and then  discard the orientation of the composed curve.

A second new curve ${(x*_P y)}_\infty$ is defined similarly as ${(x*_P y)}_0$  but orientating $x$ and $y$ so that the orientation determined by the two oriented branches emanating from $P$ disagree with the orientation of the surface.  Figure~\ref{fig:term} shows the two possible re-connections around $P$.  

\begin{figure}[ht]
  \centering
    \includegraphics[trim =60mm 65mm 40mm 35mm, clip, width=13cm]{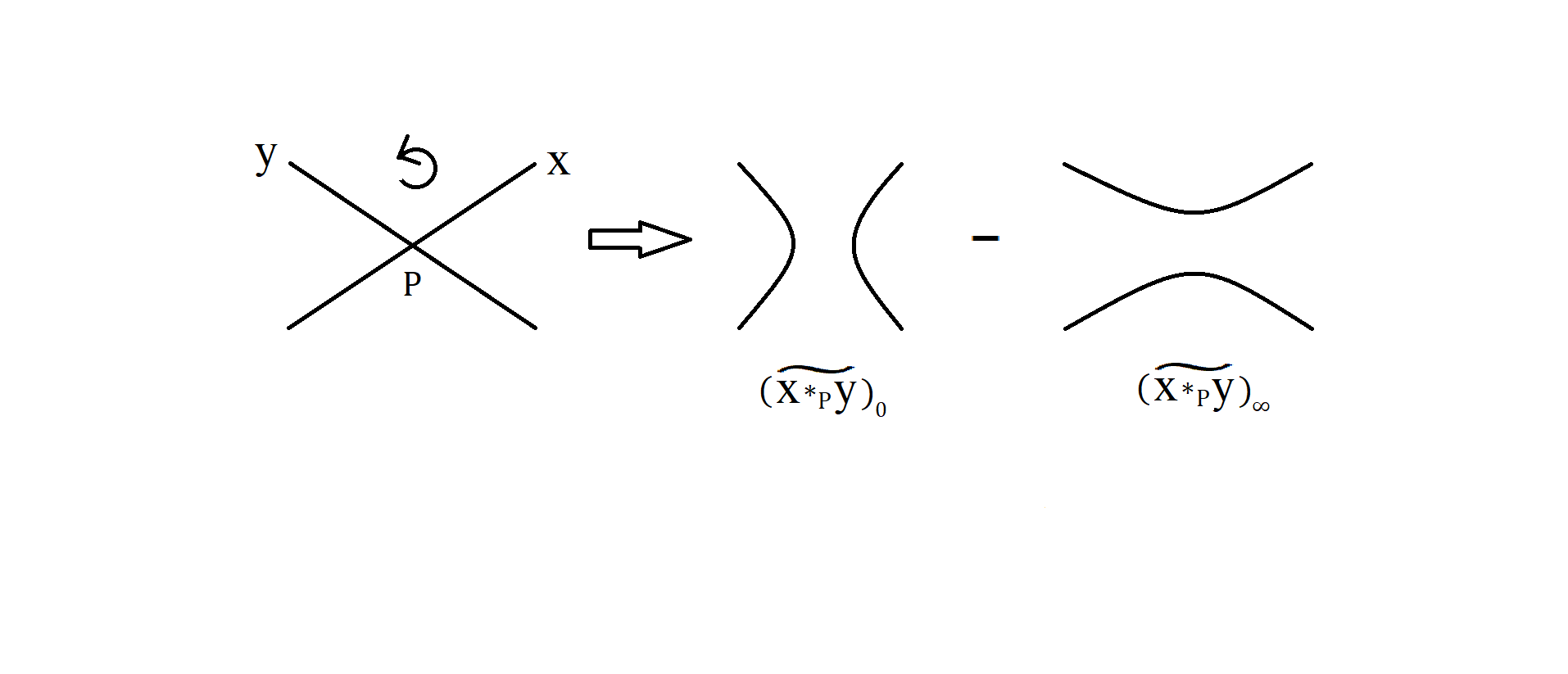}
     \caption{Local picture of the two smoothings associated with an intersection point $P$ between two undirected curves $x$ and $y$.
     }
     \label{fig:term}
\end{figure}

The free homotopy class of a curve $x$  on $\Sigma$ is denoted by $\conju{x}$. (Unless explicitly noted, all curves considered are  undirected).  Let  $\K$ be a commutative ring containing the integers. The set of free homotopy classes of undirected  closed curves  is denoted by  $\conju{\pi}$ and the free module generated by $\conju{\pi}$ over a ring $\K$ is denoted by $\K\conju{\pi}$.   We now   give the definition of a Lie algebra  on $\K\conju{\pi}$,   by giving the bracket of a pair of elements of the basis $\conju{\pi}$, and then, extending bilinearly. 

Consider two   curves $x$ and $y$ intersecting only in transversal  double points. We define geometrically $\left[\conju{\a},\conju{y}\right]$ as the sum over all the intersection points of $x$ and $y$ of the free homotopy classes of  two signed smoothings associated with each intersection point of $x$ and $y$. In symbols, 
 $$
 \left[\conju{\a},\conju{y}\right] =\sum_{P \in x \cap y } \conju{(x*_P y)}_0-\conju{(x*_P {y})}_\infty.
 $$Clearly,  if $\conju{\a}$ and $\conju{y}$ have  disjoint representatives,  $[\conju{\a},\conju{y}]=0$.
We extend the bracket linearly to $\K\conju{\pi}$ and call this structure, the  \emph{Thurston-Wolpert-Goldman Lie algebra of undirected curves, or briefly, the TWG Lie algebra}. 

 \begin{remark} Note that the smoothings at an intersection point can be defined for each pair of branches if the intersection is transversal, but the point could have multiplicity larger than two.  More about this  in Subsection~\ref{subs:intersection points}.
 \end{remark}

\begin{remark} The  definition of the Goldman Lie algebra of directed curves in \cite{goldman_invariant_1986}   was geometric. However the TWG Lie algebra of undirected curves, also defined in \cite{goldman_invariant_1986} is defined, using an algebraic argument,  as a sub Lie algebra of the Goldman Lie algebra of directed curves. See Section \ref{sec:Goldman} for details.
\end{remark}
Wolpert \cite[Theorem 4.8]{wolpert1983symplectic}, discovered a sub-Lie algebra of vector fields, \textit{the twist lattice},  on the linear span of the Fenchel-Nielsen vector fields associated to curves  (which are the infinitesimal generators of the earthquakes along these curves) and gave a topological description of this Lie algebra. Goldman  \cite{goldman_invariant_1986} showed that the twist lattice Lie algebra is the homomorphic image of a more basic Lie algebra, the TWG-Lie algebra. Namely, he proved  \cite[Theorem 5.2 and \S 5.12]{goldman_invariant_1986} that  the TWG bracket (defined above)  is  skew-symmetric and satisfies the Jacobi identity. According to Goldman, the embedding of the TWG Lie algebra in the Goldman Lie algebra -which he used to define the TWG algebra - was first observed by Dennis Johnson.

The TWG  Lie algebra, and its ``cousin", the Goldman Lie algebra (see Section~\ref{sec:Goldman} for a definition) are infinite-dimensional and still have many mathematical ``secrets" to reveal.  In this work, we make extensive use of hyperbolic geometry  to make these Lie algebras ``talk'' about their secrets.
\subsection{Main results, road maps for the proofs and an unexpected conjecture}

\subsubsection{The Center of the TWG Lie algebra}
Recall that the center of a Lie algebra on $V$ is the set of elements $v \in V$ such that $[v,w]=0$ for all $w \in V$. Etingof in \cite{etingof2006casimirs} proved using representation theory that the center of the Goldman Lie algebra of directed curves on a closed surface and coefficients in $\C$ is generated by the (class of the) trivial loop.  It is not hard to extend Etingof's result to other rings containing $\Z$. The second author  in his Ph.D. thesis,  determined that the center of the Goldman Lie algebra of curves on surfaces with boundary is generated by the trivial loop together with all curves parallel to the boundary components \cite{kabiraj2016center}. His proof treats all cases (closed surfaces or with boundary) geometrically.  

In  Section \ref{sec:center thm}, using geometric methods,  we study the center of the TWG Lie algebra of undirected curves:

\begin{namedtheorem}[Center] The center of the  TWG Lie algebra of   curves is linearly generated by the class of curves homotopic to a point,  and the classes of  curves  winding multiple times around a single  puncture or  boundary component.  
\end{namedtheorem} 

A difficulty that arises  in the study of the center using topological or geometric tools is that  formal linear combinations of classes of curves, (as opposed to single classes of curves) have to be considered.  Thus, the characterization of the center requires more argument than that of the Counting Intersection Theorem below.

Here is the idea of the proof of the Center Theorem: The intersection points of a union of geodesics $y_1, y_2,..., y_k$ and different powers of a simple geodesic $x$ are ``kind of " the same when the powers vary -they are the same ``physical points" and hence the angles are the same. If the bracket of $x^m$ with a linear combination of $y_1, y_2,..., y_k$    is zero for enough values of $m$,  then there are pairs of intersection points of $x$ with the union of $y_1, y_2,..., y_k$ that yield terms of the corresponding brackets that cancel  for different values of $m$.   This implies pairs of angles at intersection points are supplementary or congruent  for all metrics.  By twisting along the simple curve $x$, and some hyperbolic geometry we show that both possibilities lead to a contradiction.

\subsubsection{The Centers of the Universal Enveloping algebra and of the Symmetric Algebra of the TWG Lie Algebra}

In Section~\ref{sec:uea}, we extend  the method described in the above paragraph further to study the universal enveloping algebra $\U(\K\conju{\pi})$ and the symmetric algebra $\s(\K\conju{\pi})$ of the TWG Lie algebra. The universal enveloping algebra $\U(\K\conju{\pi})$ admits a natural Poisson algebra structure where the Poisson bracket is the commutator.   The symmetric algebra $\s(\K\conju{\pi})$ admits natural Poisson algebra structure induced from the TWG-Lie bracket using Leibniz rule, (see Section~\ref{sec:uea} for  definitions and references.) Using our geometric computations  together with  Poincar\'e-Birkhoff-Witt theorem, we compute the Poisson center of the each of the Poisson algebras $\U(\K\conju{\pi})$ and  $\s(\K\conju{\pi})$.

\begin{namedtheorem}[Poisson Center]
The Poisson centers of the Poisson algebras $\U(\K\conju{\pi})$ and  $\s(\K\conju{\pi})$ are generated by scalars $\K$, the free homotopy class of constant curve and the curves homotopic to boundary and punctures.
\end{namedtheorem}

Let $\Sigma$ be a closed surface. Consider $\M_{\C}^n$, the algebraic variety associated to the moduli space of representations of $\pi_1(\Sigma)$ into $GL(n, \C)$ up to conjugation, which on its smooth part admits the Goldman symplectic structure (see \cite{goldman1984symplectic}, \cite{atiyah1983yang}).

In \cite{goldman_invariant_1986}, Goldman defined the homomorphism of Poisson algebras $\s(\C{\pi})\rightarrow \C[\M_{\C}^n]$ defined by $\Phi_n({x})=tr_x$  (here, $\C{\pi}$ denotes the module spanned by the set $\pi$ of conjugacy classes of  fundamental group of surface - that is, the set of free homotopy classes of \textit{directed} curves- and $\s(\C{\pi})$ denotes the symmetric algebra of $\C{\pi}$). The map $\Phi_n$  is surjective \cite{etingof2006casimirs} but not injective in general. However Etingof \cite[Proposition 2.2]{etingof2006casimirs} proved that given any finite dimensional subspace $V$ of $\C\pi$, there exists $N$ such that $\Phi_n|_V$  is injective for all $n\geq N$. Using this result together with the fact that Poisson center of $\C[\M_{\C}^n]$ being $\C$, Etingof computed the center of $\C\pi$. In principle Etingof's method could be used to compute the center of $\C\conju{\pi}$ by replacing $GL_n(\C)$ in the definition of $\M_{\C}^n$ by one of the following groups: $O_n(\C), Sp_n(\C)$ or $U_n(\C)$   \cite[Theorem 3.14 and Theorem 5.13]{goldman_invariant_1986}. However, as with the Etingof's proof, this possible method will work only for closed surfaces. Our proof, on the other hand, works for any complete hyperbolic surface.

 Chas and Gadgil \cite{chas2016extended} used geometric group theory to study the quasi-geodesic nature of the lifts of the terms of the Goldman bracket. Using this result and the facts that lifts of a simple closed curve are disjoint, in \cite{kabiraj2016center}, the second author computed the center of $\K\pi$. Although,  this method could be used to compute the center of $\K\conju{\pi}$, there are two drawbacks. Firstly because the proof is based on case by case considerations, it is long, technical and geometrically less transparent than our proof here. Secondly because of the geometric group theory techniques, the various bounds obtained would be qualitative and not quantitative. 
 
\subsubsection{Canonical decomposition of the Goldman Lie algebra and the TWG Lie algebra}

Denote by $\text{C}$  (resp. $\widetilde{\text{C}}$) the free homotopy classes of the directed (resp. undirected) trivial curve, and the curves or powers of curves parallel to boundary components. (Note that this    $\text{C}$  and. $\widetilde{\text{C}}$ are a basis of the center of the Goldman and TWG Lie algebras respectively). Denote by 
$\pi_0$ (resp. $\conju{\pi}_0$)  the set of free homotopy classes of  directed (resp. undirected) closed curves minus  $\text{C}$  (resp. $\widetilde{\text{C}}$). 

Goldman \cite{goldman_invariant_1986} stated that  for closed surfaces,
the Goldman Lie algebra admits a canonical decomposition $\mathbb{Z} \text{O} \oplus\mathbb{Z}{\pi}_0$, where $\text{O}$ represents the class of constant curve. The result is valid, but the proof has a gap, namely, it is not true that if $\alpha$ is the class of a directed curve, and $\bar{\alpha}$ is the class of  $\alpha$ with opposite direction, then the Goldman bracket of $\alpha$ and $\bar{\alpha}$, $[\alpha,\bar{\alpha}]$ is zero. For instance, if $\alpha$ is a figure eight curve in the pair of pants (that goes around  two boundary components) then  $[\alpha,\bar{\alpha}]$  has two terms that do not cancel.  On the contrary, the first author has conjectured in  \cite{chas2004combinatorial}  that for all directed curves $\beta$ the number of terms of  $[\beta,\bar{\beta}]$  (counted with multiplicity) is twice the self-intersection number of $\beta$. This conjecture is a ``pre-theorem'' in the sense that it is strongly supported by computer evidence.

Our techniques yield  a proof of 
Goldman's direct sum statement about the Lie algebra of directed curves, and also, the analogous result for the TWG bracket. 

\begin{namedtheorem}[Canonical Decompositions] The Goldman bracket of two non-central directed curve classes, expressed in the natural basis,  does not contain a central element as a term. The analogous result holds for the TWG bracket. \\
(1) The Goldman Lie algebra admits a Lie algebra decomposition $\K\mathrm{C}\oplus\K \pi_0$.\\
(2) The TWG Lie algebra also admits a Lie algebra decomposition 
$\K\widetilde{\mathrm{C}}\oplus \K \conju{\pi}_0$.
\end{namedtheorem}

\subsubsection{A new proof of the Counting Intersection Theorem}

 

 Given two free homotopy classes of   closed curves $\conju{x}$ and $\conju{y}$, the \emph{geometric intersection number of $\conju{x}$ and $\conju{y}$}, denoted by $i(\conju{x},\conju{y})$, is defined to be the smallest number of crossings of a pair of representatives $\conju{x}$ and $\conju{y}$, that intersect in transversal double points.  Goldman \cite[Theorem 5.7]{goldman_invariant_1986}  proved that if the TWG bracket of two  classes of undirected curves is zero, and one of them has a simple representative, then the classes contain disjoint representatives. He also prove the analogous result for directed curves.  Chas  \cite{chas2010minimal} using combinatorial group theory generalized this result by proving that the Goldman and the TWG bracket of two curves, one of them simple, counts  intersection number.  In both cases, the main tool was the use of  free products with amalgamation and HNN structures on the fundamental group of the surface, determined by simple curves. In the case of the TWG- bracket, Chas's proof was complicated requiring  several cases and combinatorial technical lemmas. Our techniques allow us to give a new geometric proof that the TWG bracket counts intersections 

\begin{namedtheorem}[Counting Intersection] If $\a$ and $\b$ are closed curves and $\a$ is simple then   the number of terms (counted with multiplicity) of $[\conju{\a},\conju{\b}]$ is twice the intersection number,  $2i(\conju{\a},\conju{\b})$.  
\end{namedtheorem}

One noteworthy point about our geometric definition of the TWG Lie bracket here and the resulting angle technique is a new  
proof in two straightforward statements proving no cancellation of terms in the TWG bracket: The first of these statements is that  cancellation means that a certain pair of angles are supplementary; the second, is that earthquaking along the simple curve changes both angles strictly monotonically (both decreasing or both increasing). This is a contradiction, proving the result. 


\subsubsection{A computational result and conjectural  characterization of disjointness of closed curves via the TWG bracket}

The Counting Intersection Theorem  is only valid if  of the curves is  simple, see \ref{ex:non-simple}. However, computer experiments suggest that the TWG bracket actually detects disjointness in all cases.  Thus we conjecture:

\begin{conj*} If the TWG Lie bracket of two distinct classes of undirected curves is zero, then the classes have disjoint representatives.
\end{conj*}

The above statement was  verified computationally for as many classes of curves available computers can handle (see Section~\ref{sec:concrete} for precise statements).
This conjecture is unexpected, given that the equivalent statement does not hold  for the Goldman Lie bracket on directed curves. In other words, there are examples of pairs of classes of directed curves with Goldman bracket zero and without disjoint representatives, see for instance  Example~\ref{ex:non-simple directed}.

There are also examples that show  the TWG Lie bracket of two non-simple curves can have cancellation, see  Example \ref{ex:non-simple}. Thus, the hypothesis of one of the curves being simple in  the Counting Intersection Theorem cannot be dropped in general.

\begin{center}
\begin{tabular}{ |c|c|c|c|c| } 
 \hline
 Bracket &$[x,y]=0$, $x$ simple   
 & $x$ simple then  $[x,y]$
 &$[x,y]=0$ implies & $[x,y]$ counts \\
 & implies disjointness 
  & counts intersection 
 &  disjointness &  intersection \\
 \hline
 TWG & Yes &Yes & Yes, if conjecture holds & No \\ 
 Goldman   &Yes &Yes &No & No \\ 
 
 \hline
\end{tabular}
\end{center}

\subsubsection{Direct proof of the Jacobi identity for the TWG bracket}
In Section~\ref{sec:Goldman} we prove that our definition and that of Goldman \cite{goldman1984symplectic} coincide. 
For completeness, in Section~\ref{sec:jacobi}, we prove that the TWG bracket satisfies the Jacobi identity.

\subsection{Applications}

In \cite{turaev1991skein}, Turaev introduced various skein modules associated to the set of isotopy classes of links in the  three manifold $\Sigma\times I$ for the quantization of $\U(\K\conju{\pi})$ and  $\s(\K\conju{\pi})$.
In \cite{hoste1990homotopy}, Hoste and Przytycki  independently gave another quantization of the same in terms of \emph{homotopy skein modules}. 
Our results about the center of $\U(\K\conju{\pi})$ and  $\s(\K\conju{\pi})$  can be used to  compute the center of these skein modules. See \cite{kabiraj2020poisson}.

\subsection{Other results relating the intersection structure of curves with the Goldman and TWG bracket }

In \cite{turaev1991skein}, Turaev discovered cobracket operations which gives Goldman algebra and TWG Lie algebra the structure of a Lie bialgebra.  Trying to understand the relation of the Goldman-Turaev Lie bialgebra with intersection and self-intersection of curves on surfaces Chas and Sullivan discovered String Topology \cite{chas1999string}, a structure that generalizes the Goldman Lie algebra and the Turaev Lie coalgebra to arbitrary orientable manifolds of all dimensions.

Chas and Krongold \cite{chas2010algebraic} proved that, on a surface with boundary,  a non-power directed curve $x$  is simple   if and only if the Goldman Lie bracket of $x$ with $x^m$ is non-zero, for  any  $m \ge 3$. Moreover, for any directed curve $x$, the number of terms, counted with multiplicity, of the Goldman  Lie bracket of $x$ with $x^m$ is  $2m$-times the geometric self-intersection number of $x$.  (Computer evidence suggests these statements are also true when  $m=2$. We are working on a proof of this result.) 

Chas and Gadgil \cite{chas2016extended} proved that  if $x$ and $y$ are non-power directed  curves, then for $p$ and $q$ large enough, the Goldman Lie  bracket $[{x^p},{y^q}]$ counts the geometric intersection number of the classes of $x$ and $y$. 

Cahn and Tchernov \cite{cahn2013intersections}   determined that the   Andersen-Mattes-Reshetikhin Poisson bracket (a generalization of the Goldman Lie  bracket) counts intersections of two classes of curves, when the classes are distinct.

Kawazumi and Kuno \cite{kawazumi2012center} proved that the center of the Goldman Lie algebra on directed curves on a surface with infinite genus and one boundary component is spanned by constant loop and powers of loops parallel to the boundary component. 

Recently in \cite{turaev2019loops}, Turaev gave a new description of Goldman Lie algebra using \emph{star-fillings} (also see \cite{turaev2019topological}).

In both the Turaev cobracket \cite{turaev1991skein} and the Chas-Sullivan \cite{chas1999string} efforts the motivation to understand embedded curves was related to the Jaco-Stallings \cite{stallings1966not} sixties equivalence between a statement about embedded curves and the Poincar\'e conjecture. Since this statement about embedded curves a five hundred pages treatrise by Morgan and Tian \cite{morgan2007ricci} of Perelman's work it is even more compelling to give a geometric proof (hopefully, in substantially less than five hundred pages) in the language of curves and surfaces. The papers mentioned show this is a vast subject.

\subsection{Organization of the paper}
 This work is organized as follows: In Section~\ref{sec:prelim} we recall some results of hyperbolic geometry and give the description of  a lift of the terms of TWG Lie bracket to the upper half plane $\H$. We also prove the Counting intersection Theorem in Section~\ref{sec:prelim}. In Section~\ref{sec:center thm} we prove the Center Theorem. In Section~\ref{sec:uea} we compute the Poisson center of the universal enveloping algebra and symmetric algebra of TWG Lie algebra. In Section~\ref{sec:concrete} we give explicit examples of computations of Goldman Lie bracket an TWG Lie bracket of curves in some low complexity hyperbolic surfaces and state a conjecture about TWG Lie bracket. In Section~\ref{sec:Goldman}, we prove that our definition coincides with that of Goldman \cite{goldman_invariant_1986}. In particular, this implies that the TWG algebra is well defined on homotopy classes of undirected curves and it is indeed a Lie algebra. For a proof of Jacobi's identity for TWG bracket see Section \ref{sec:jacobi}. In Section~\ref{sec:candec} we prove the Canonical Decomposition Theorem.  
\section*{Acknowledgements}
This work benefited from communications with Vladimir Turaev, Scott Wolpert and William Goldman. 
The first author was partially supported by the NSF. The second author was supported by the DST, India: INSPIRE Faculty fellowship. 

\section{Intersection points, angles and earthquakes. Proof of the Counting Intersection Theorem}\label{sec:prelim}

Denote by $\T$ the Teichm{\"u}ller space associated with the surface $\Sigma$. A closed curve on $\Sigma$ is an \textit{$X$-geodesic} if it is a geodesic for the metric $X \in \T$.

\subsection{Intersection points and metrics}\label{subs:intersection points}

In order to be able to follow intersection points of two curves through homotopies of these curves, we need to refine the definition of intersection point: If $x$ and $y$ are two closed curves intersecting transversally, an \textit{$(x,y)$-intersection point} is a point $P$ on the intersection of $x$ and $y$, together with a choice of a pair of small arcs, one of $x$ and the other of $y$, intersecting only at $P$. 

\begin{remark} For any two curves $x$, $y$ (possibly with intersection points of multiplicity larger than two)  we have  $$i(\conju{x},\conju{y})=\mathrm{min}\#\{\xy : x\in \conju{x}, y\in\conju{y} \mbox{, $x$ and $y$ intersect transversally}\}.$$ Here the number of intersection points counted with multiplicity, namely $k$ lines intersecting transversally at a point counts as  ``$k$ choose two".
\end{remark}

By Thurston's Earthquake Theorem (see the appendix of \cite{kerckhoff1983nielsen} for a proof), there is a unique earthquake path between any pair $X, X'$ of elements of $\T$. The next lemma can be proved by following the $\xy$ of two $X$-geodesics $x$ and $y$ along this unique geodesic path.

\begin{lemma}\label{lem:correspondence} Let $X, X' \in \T$ and let $x$ and $y$ be two $X$-geodesics. If $x'$, $y'$ are two $X'$-geodesics such that both pairs $x, x'$ and $y, y'$ are homotopic then   there is a canonical bijection between the $\xy$ and the $\xyp$.
\end{lemma}
\begin{remark} Lemma~\ref{lem:correspondence} 
can be also proved using that $\T$ is simply connected, instead of Thurston's Earthquake Theorem.
\end{remark}

\subsection{Angles, metrics and earthquakes }

Fix a metric $X \in \T$ and $P$, an  $(x,y)$-intersection point of two  $X$-geodesics  $x$ and $y$. For each metric $Y \in \T$, 
the \emph{$Y$-  angle of $x$ and $y$ at $P$}, denoted  by $\phi_P(Y)$, is  defined as the angle at the intersection point corresponding to $P$ by Lemma~\ref{lem:correspondence}, of the two $Y$-geodesics homotopic to $x$ and $y$, measured from the geodesic homotopic to $y$ to the geodesic homotopic to $x$, following the orientation of the surface (see Figure~\ref{fig:phi}).  Clearly,  $\phi_P(X)$ is the angle at $P$, from $y$ to $x$ following the orientation of the surface.
 Observe that   $\phi_P(Y) \in (0,\pi)$. (The angle $\phi_P$ is defined from $y$ to $x$ as is done in \cite{wolpert_fenchel-nielsen_1982}).

\begin{figure}[ht]
  \centering
    \includegraphics[trim =65mm 95mm 10mm 45mm, clip, width=21cm]{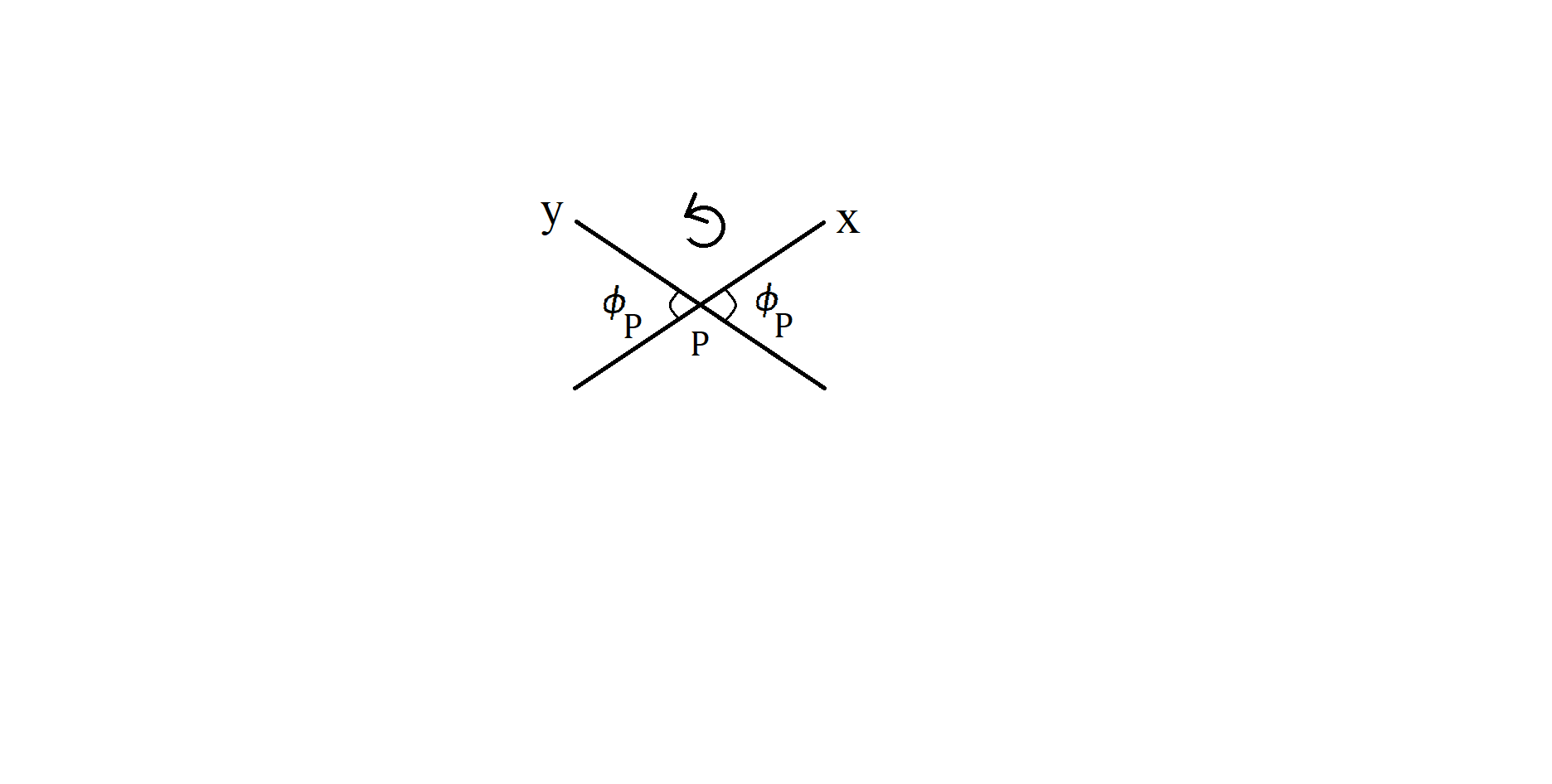}
     \caption{The   angle $\phi_P$
     }
     \label{fig:phi}
\end{figure}
Following \cite{kerckhoff1983nielsen},  for each simple   $X$-geodesic $x$ and each real number $t$,  $\mathcal{E}_{x}(t)$, is the element of $\T$ given by left twist deformation of $X$ along $x$ at the time $t$ starting at $X$.  
 (Clearly, $\mathcal{E}_x(t)$ also depends  on $X$).

 By \cite[Proposition 3.5]{kerckhoff1983nielsen}  and  \cite[Lemma 2.1]{kabiraj2018equal} we have,

\begin{lemma}\label{lem:decreasing} If $X \in \T$, and $x$ and $y$ are   two $X$-geodesics such that $x$ is simple, and $P$ is an $(x,y)$-intersection point then the function $\phi_P(\mathcal{E}_x(t))$ is a strictly decreasing function of $t$. 
\end{lemma}

\begin{remark} The aforementioned  Proposition 3.5 of \cite{kerckhoff1983nielsen} is  proved assuming that both $x$ and $y$ are simple geodesics. In the same work,  it is stated that it holds when $y$ is  non-simple. An explicit proof of this fact can be found in \cite[Lemma 2.1]{kabiraj2018equal}. 
\end{remark}
We include the following result from \cite[Theorem 7.38.6]{beardon2012geometry} and part of its proof because both will be used later.  
\begin{theorem}\label{theo:beardon} Let $a$ and $b$ be  hyperbolic isometries of the hyperbolic plane, whose axes  intersect at a point $P$. Denote by $\beta$  the angle at $P$ of these axes in the forward direction of $a$ and $b$.  Then the product $a.b$ is hyperbolic and
$$\cosh\left(\frac{t_{a.b}}{2}\right)=
\cosh\left(\frac{t_a}{2}\right)\cosh\left(\frac{t_b}{2}\right)+
\sinh\left(\frac{t_b}{2}\right)\sinh\left(\frac{t_b}{2}\right)\cos(\beta),$$
where $t_\alpha$ denotes the translation length of $\alpha$.
\end{theorem}
\begin{proof}
Denote by $Q$ the point on the axis of $a$ at distance $t_a/2$ of $P$ in the positive direction of the axis of  $a$ and by $R$ the point on the axis of $b$ at distance $t_b/2$ of $P$ in the negative direction of the axis of $b$. The axis of $a.b$ is the geodesic containing the oriented line from $R$ to $Q$ and the translation length of $a.b$ equals twice the distance between $R$ and $Q$. (see Figure~\ref{fig:beardon}).
The formula follows from the Cosine formula (see \cite{beardon2012geometry} for more details). 
\begin{figure}[ht]
  \centering
    \includegraphics[trim =35mm 65mm 10mm 10mm, clip, width=13cm]{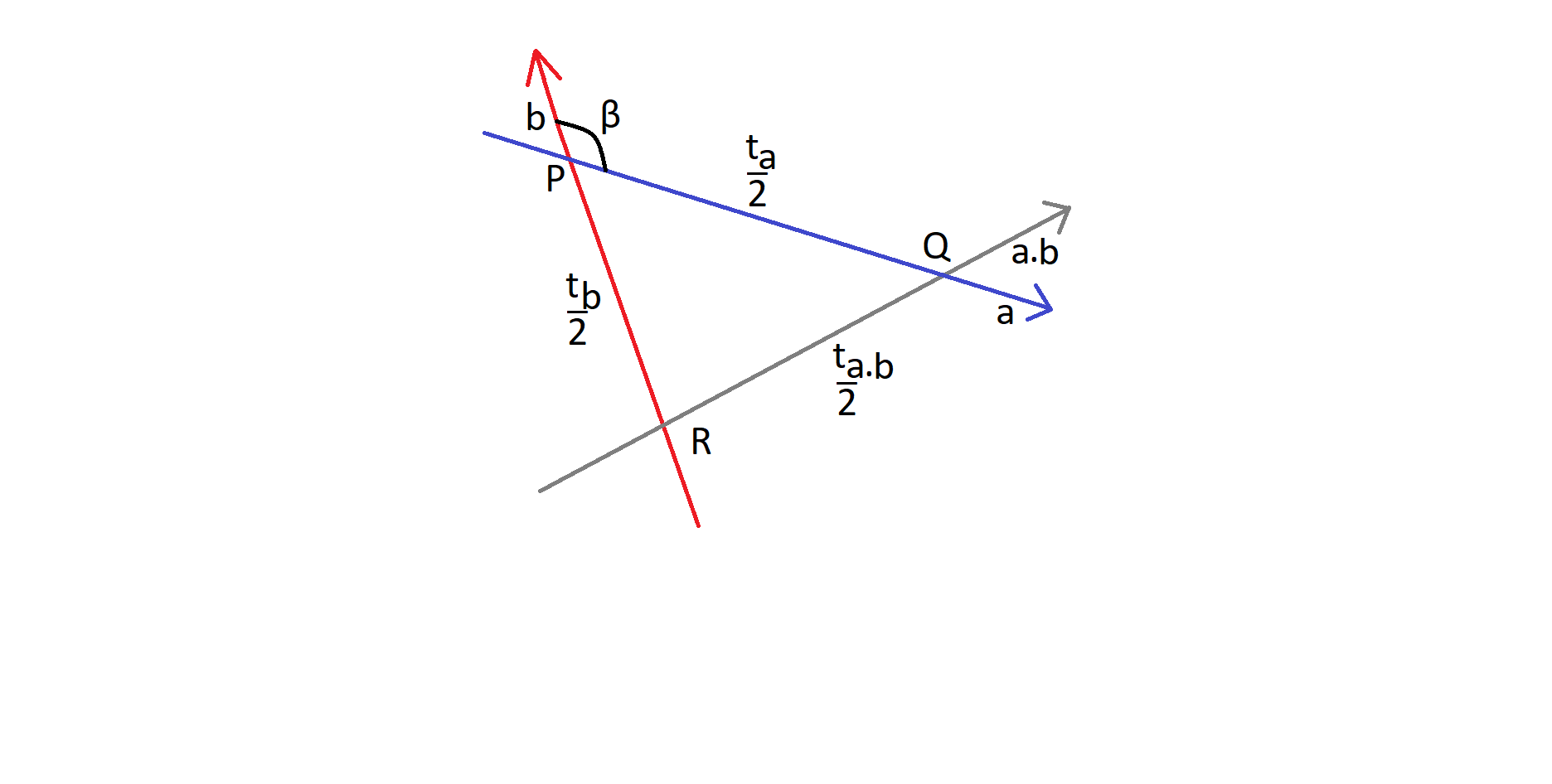}
     \caption{Theorem~\ref{theo:beardon}}\label{fig:beardon}
\end{figure}
\end{proof}
For each closed curve $x$ on $\Sigma$, the length of the unique $X$-geodesic in $\conju{x}$ is denoted by  $\ell_x(X)$.

\begin{figure}[ht]
  \centering
    \includegraphics[trim =75mm 45mm 10mm 40mm, clip, width=13cm]{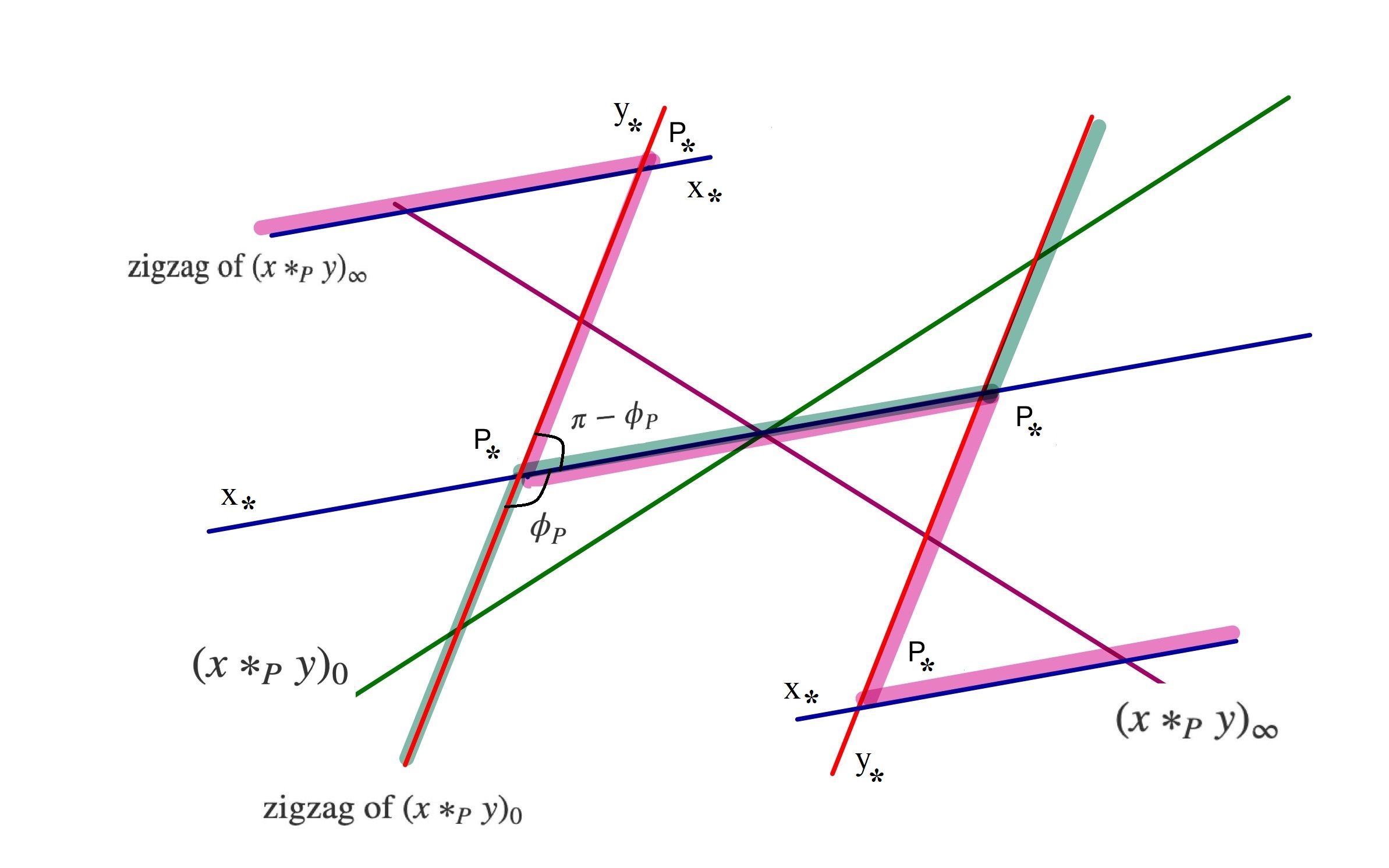}
     \caption{Lifts of ${(x*_p y)}_0$  and  ${(x*_p y)}_\infty$. }\label{fig:lifts}
\end{figure}

Given two $X$-geodesics $x$ and $y$ and an $(x,y)$-intersection point $P$, a lift of ${(x*_p y)}_0$ (respectively of ${(x*_p y)}_\infty$) to the upper half plane $\H$  is a bi-infinite piecewise geodesic (see Figure \ref{fig:lifts}) consist of alternative geodesic segments of lift of the geodesics $x$  and $y$ (denoted by $x_*$ and $y_*$ in Figure \ref{fig:lifts} respectively). The  geodesic segments of lift of the geodesics $x$  and $y$ intersect each other in the lifts of the point $P$ (denoted by $P_*$ in Figure \ref{fig:lifts}).

 The next lemma follows directly from Theorem~\ref{theo:beardon}, and from Figure \ref{fig:lifts} by adding an appropriate orientation to the geodesics $x$ and $y$.
 
 \begin{lemma}\label{lem:cosh} If $x$ and $y$ are two closed $X$-geodesics and let $P$ be an $(x,y)$-intersection point.  Then we have
 $$\cosh\left(\frac{\ell_{(x*_P y)_0}}{2}\right)=
\cosh\left(\frac{\ell_x}{2}\right)\cosh\left(\frac{\ell_y}{2}\right)-\sinh\left(\frac{\ell_x}{2}\right)\sinh\left(\frac{\ell_y}{2}\right)\cos(\phi_P)$$
$$\cosh\left(\frac{\ell_{(x*_P y)_\infty}}{2}\right)=
\cosh\left(\frac{\ell_x}{2}\right)\cosh\left(\frac{\ell_y}{2}\right)+\sinh\left(\frac{\ell_x}{2}\right)\sinh\left(\frac{\ell_y}{2}\right)\cos(\phi_P),$$
where all lengths and angles are computed with respect to any metric $Y$ in $\T$. 
 \end{lemma}

The next lemma states that, a $0$-term of the bracket of two curves, and an $\infty$-term of the bracket of the same curves are always distinct when one of them is simple. It is easier to state it in terms of geodesics (instead of free homotopy classes of curves). 
 
\begin{lemma}\label{lem:zero-infinty} Let $X$ be a hyperbolic metric on $\Sigma$. If $x$ and $y$ are closed,   $X$-geodesics, such that $x$ is simple, and $P$ and $Q$ are two (not necessarily distinct) $(x,y)$-intersection points then $\conju{( x*_P y)_0} \ne \conju{( {x}*_Q {y})_\infty}.$
\end{lemma}
\begin{proof} 
We argue by contradiction. If $\conju{( x*_P y)_0} = \conju{( {x}*_Q {y})_\infty}$ then $\ell_{( x*_P y)_0}(Y) = \ell_{( {x}*_Q {y})_\infty}(Y)$ for any $Y\in\T$.
By Lemma~\ref{lem:cosh} we have   $$\cos(\phi_{P}(Y))=-\cos(\phi_{Q}(Y)),$$ which implies,  
\begin{equation}
\phi_{P}(Y)+\phi_{Q}(Y)=\pi.\label{eq:supplementary}
\end{equation}
for all $Y \in \T$. 
On the other hand, as $x$ is simple, by Lemma~\ref{lem:decreasing}, by twisting the metric $X$  about the geodesic  $x$, both terms on the right side of Equation~(\ref{eq:supplementary}) strictly decrease. Since they add up to a constant, this is  not possible. Hence, the proof is complete.
\end{proof}

\subsection{Proof of the counting intersection theorem}

If  $\conju{x}$ and $\conju{y}$ have disjoint representatives, the result follows directly. Assume that $i(\conju{x}, \conju{y}) >0$.  From the definition of the bracket,  it follows that
$$\left[\conju{x},\conju{y}\right] =\sum\limits_{P \in x \cap y } \conju{(x*_P y)}_0-\conju{(x*_P {y})}_\infty.$$Fix a metric $X \in \mathcal{T}$.    In order to simplify the notation, assume that $x$ and $y$ are $X$-geodesics. This implies that $x$ and $y$ intersect  in $i(\conju{x}, \conju{y})$ points,  the geometric intersection number of the class. 

Suppose that the number of terms of the bracket is strictly smaller than $2.i(\conju{x},\conju{y})$. Hence,  there  exist two (not necessarily distinct)   (x,y)-intersection points $P$ and $Q$  such that a pair of terms corresponding $P$  and $Q$ cancel. 

The terms corresponding to $P$ are  $\conju{( x*_P y)_0} - \conju{( {x}*_P {y})_\infty}$
and the terms corresponding to $Q$ are $\conju{( x*_Q y)}_0 - \conju{( {x}*_Q y)}_\infty.$ 

The assumption of cancellation implies that  either   $\conju{( x*_P y)}_0=\conju{( x*_Q y)}_\infty$ or 
$\conju{( x*_P y)}_\infty=\conju{( x*_Q y)}_0$
which is not possible by Lemma~\ref{lem:zero-infinty}. Thus, the proof is complete.

\begin{remark}
In the  case  of the Goldman bracket of two directed curves, cancellation of two terms (regardless whether they are simple or not) implies that the corresponding directed angles are congruent, \cite[Theorem 5.1]{kabiraj2018equal}. (The directed angle between two directed geodesics intersecting at a point $P$ is the angle between the positive direction of both curves).

In the case of the TWG-bracket, cancellation of two terms implies that the (undirected)  angles are supplementary.\end{remark}

\section{TWG-Lie Bracket of powers of curves and proof of the Center Theorem}\label{sec:center thm}

Let $X \in \T$. If $P$ is an $(x,y)$-intersection point of two $X$-geodesics $x$ and $y$ then for each positive integer $m$, the geodesic $x^m$ (that goes  $m$ times around $x$) and $y$ also intersect at $P$. 

\begin{remark}\label{rem:powers} The angles at $P$ of $x$ and $y$ and of $x^n$ and $y$ are congruent (they are the same angle). Thus, we can (and will) consider $P$ as $(x^m,y)$-intersection point and the angle $\phi_P(Y)$ will also denote the angle at $P$ of the $Y$-geodesic homotopic to $x^m$ and the $Y$-geodesic homotopic to $y$.
\end{remark}

\begin{proposition}\label{prop:values of m}	Let $X \in \T$ be hyperbolic metric on  $\Sigma$. Let $x,y$ and  $z$ be three  $X$-geodesics.  Let  $P$ and $Q$ be two $(x,y)$ and $(x,z)$-intersection points respectively.  
\begin{enumerate}
\item      If there exist two distinct positive values of $m$ such that  $\conju{(x^m*_Py)}_0=\conju{(x^m*_Qz)}_0$   (resp. $\conju{(x^m*_Py)}_\infty=\conju{(x^m*_Qz)}_\infty$  )   then $\ell_y(Y)=\ell_z(Y)$ and $\phi_P(Y)=\phi_Q(Y),$ for all $Y \in \T.$      
\item      If there exist two distinct positive values of $m$ such that  $\conju{(x^m*_Py)}_0=\conju{(x^m*_Qz)}_\infty$   then $\ell_y(Y)=\ell_z(Y)$ and $\phi_P(Y)+\phi_Q(Y)=\pi$ for all $Y \in \T$.
\end{enumerate}
\end{proposition}

\begin{proof} 
We prove  (1); the proof of  (2) is analogous. 
From now on, we will fix a metric $Y \in \T$. To simplify the notation, we will not write the dependence on $Y$ (for instance, we will write $\cos(\phi_P)$ instead of $\cos(\phi_P)(Y)$).   We follow  the notation indicated in Remark~\ref{rem:powers}. 

Since $\conju{(x^m*_Py)}_0=\conju{(x^m*_Qz)}_0,$ $\ell_{(x^m*_Py)_0}= \ell_{(x^m*_Qz)_0}.$ By Lemma \ref{lem:cosh}, we have
$$\cosh(\frac{1}{2} \ell_{{x^m}})\cosh(\frac{1}{2} \ell_{{y}})-\sinh(\frac{1}{2} \ell_{{x^m}})\sinh(\frac{1}{2} \ell_{{y}})\cos(\phi_P)=$$$$\cosh(\frac{1}{2} \ell_{{x^m}})\cosh(\frac{1}{2} \ell_{{z}})-\sinh(\frac{1}{2} \ell_{{x^m}})\sinh(\frac{1}{2} \ell_{{z}})\cos(\phi_Q).$$
	
This implies 
$$ \coth(\frac{1}{2} \ell_{{x^m}})\{\cosh(\frac{1}{2} \ell_{{y}})-\cosh(\frac{1}{2} \ell_{{z}})\}=\sinh(\frac{1}{2} \ell_{{y}})\cos(\phi_P)- \sinh(\frac{1}{2} \ell_{{z}})\cos(\phi_Q)	
$$
Note that that the right-hand side of the above equation does not depend on $m$. Also, if $m_1$ and $m_2$ are distinct positive integers then $\coth(\frac{1}{2} \ell_{x^{m_1}}) \ne \coth(\frac{1}{2} \ell_{x^{m_2}}).$ This implies $\cosh(\frac{1}{2} \ell_{{y}})-\cosh(\frac{1}{2} \ell_{{z}})=0,$ and so,    $\ell_y=\ell_z$.  Hence, $\cos(\phi_P)=\cos(\phi_Q),$ which implies the equality of the corresponding angles, as desired.
\end{proof}

\begin{lemma}\label{lem:smaller angle}
Let $X$ in $\T$ and let $x,y$ and $z$ be three  $X$-geodesics in $\Sigma$ such that $x$ is simple,  $\ell_y(Y)=\ell_z(Y)$ for all $Y \in \T$ and there exist an $(x,y)$-intersection point $P$ and a $(x,z)$-intersection point $Q$ such that for some positive integer $m$,  $\conju{(x^m*_P y)}_0=\conju{(x^m*_Q z)}_0$. Then 
either       $y=z$   or there exist an $(x,y)$-intersection point $R$ such that  $\phi_R>\phi_P.$
\end{lemma}
\begin{proof}
We prove the result for $m=1$. 
Combining Remark~\ref{rem:powers} with the equality  $\ell_{x^m}=m\ell_x$ the proof for $m>1$ follows by a similar argument.

Since  $\conju{(x*_P y)_0}=\conju{(x*_Q z)_0}$,  there exists two lifts to the universal cover of the surface, the hyperbolic plane $\H$, one of  the piecewise geodesics  $(x*_P y)_0$  and the other of  $(x*_Q z)_0$ with the same endpoints. Denote these two lifts by $C$ and $D$ respectively, and by $L$ the geodesic line joining their common endpoints. 
 The two piecewise geodesics $C$ and $D$   zigzag about the line $L$. The zigzag curve $C$ (resp. $D$)  is composed  of alternating segments of lifts of $x$ of length $\ell_x$,  and lifts of $y$ (resp. of $z$) of length $\ell_y$.  
 In  Figure~\ref{fig:zigzag},  ``laps'' of lifts of $x$ are represented in blue, ``laps'' of  lifts of $y$ in green and ``laps'' of lifts of $z$ in red. The line $L$ intersects each of these lap segments in their midpoints. (See Theorem~\ref{theo:beardon} and its proof).
 
 Consider a  segment $S$ of the zigzag curve $C$, which is a  lift of $x$. 
 Denote the intersection point of $S$ and $L$ by $U$. 
 Choose an endpoint of $S$ and denote it by $P_1$. 
 Denote by $V$ the intersection of $L$ with the other  segment of $C$ with endpoint $P_1$ (this last segment is a lift of $y$).
 These three points determine a triangle $UVP_1$
 
 Consider the triangle $U'V'Q_1$, analogous to $UVP_1$, but with sides included in the zigzag curve $D$. 
 
 Note that the length of both segments,  $UP_1$ and  $UQ_1$ is $\ell_x/2$. Also the length of $VP_1$ and $V'Q_1$ is $\ell_y/2$. By Lemma~\ref{lem:cosh}, the length of $UV$ and $U'V'$ is half the length of the geodesic in $\conju{(x*_P y)_0}$. Therefore, these two triangles $UVP_1$ and $U'V'Q_1$ are congruent. 
 Hence,  there is an isometry mapping one triangle to the other. If this isometry is orientation reversing, then it maps 
 $U$  to $U$, $V$ to $V'$ and $P_1$ to $Q_1$. This implies that $\phi_P(X)+\phi_Q(X)=\pi$, see Figure~\ref{fig:zigzag}, a. 
 
 If we perturb the metric $X$ slightly, we can repeat the above argument, and obtain that  $\phi_P(Y)+\phi_Q(Y)=\pi$, for all $Y$ in a neighborhood of $X$. (The orientation reversing isometry for the corresponding $Y$-geodesics must exist by continuity). Since $x$ is simple, this is not possible by  Lemma \ref{lem:decreasing}.
  Thus, there is an orientation preserving isometry mapping  $U$  to $U$, $V$ to $V'$ and $P_1$ to $Q_1$. Now, there are two possibilities: either the midpoint of a lap of a lift of $x$ is also a midpoint of lap of a lift or $y$ (Figure~\ref{fig:zigzag}, right) or not (Figure~\ref{fig:zigzag}, middle.)
 
 \begin{figure}[ht]
  \centering
    \includegraphics[trim =80mm 35mm 40mm 5mm, clip, width=14cm]{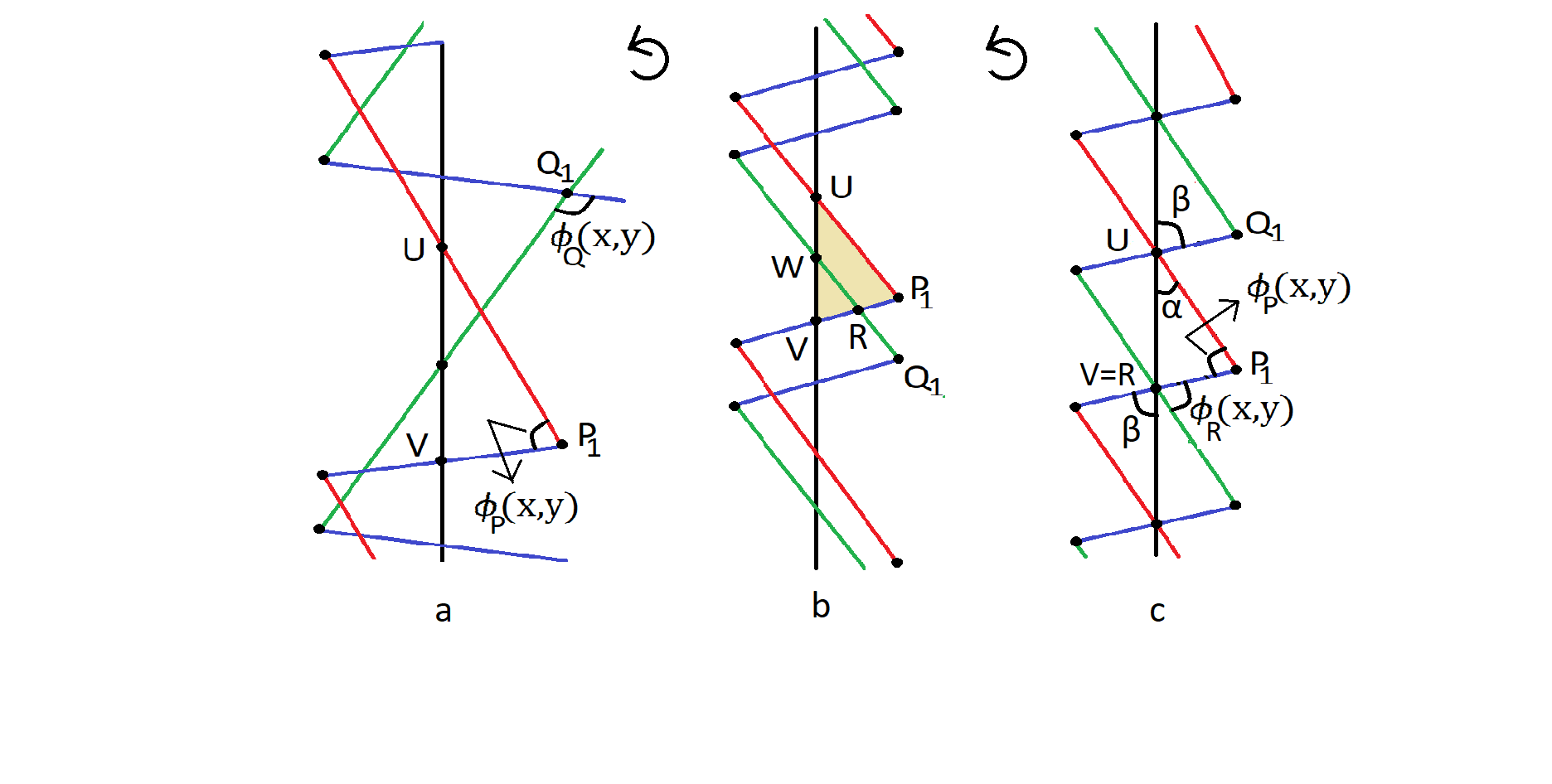}
     \caption{Zigzags}\label{fig:zigzag}
\end{figure}

If $C=D$, then $y=z$ and the proof is complete. Hence, we can assume $C\ne D$. There are then two  cases left, depicted in  Figure~\ref{fig:zigzag}, b. and c. 
In the case illustrated in Figure~\ref{fig:zigzag}, b., a segment lifting of $z$ intersects the interior of the triangle $UVP_1$, and determines a triangle $WVR$ as in the figure. Since the area of $WVR$  is smaller than that of $UVP_1$, and two of the angles of of $WVR$  are congruent to two of the angles of   $UVP_1$, the angle at $P_1$, is smaller than the angle at $R$, so the proof of this case is complete.

In the case illustrated in Figure~\ref{fig:zigzag},  c., $\phi_P+\alpha+\beta < \pi  = \phi_R+\alpha + \beta$, and $\phi_P < \phi_R$, as desired.
\end{proof}

\begin{proposition}\label{prop:noncapower}
	Let $X\in\T$ be a metric on $\Sigma$ and $\a,y,z$ be three pairwise distinct $X$-geodesics such that $\a$ is simple and let $P$ and $Q$ be  $(x,y)$ and $(x,z)$-intersection points respectively. The following holds.
	\begin{enumerate}
		\item  The equality  $\conju{(\a^m*_P y)}_0=\conju{(\a^m*_Q z)}_\infty$ holds for at most one positive value of $m$.
		\item  Either the equality $\conju{(\a^m*_P y)}_0=\conju{(\a^m*_Q z)}_0$ (resp. $\conju{(\a^m*_P y)}_\infty=\conju{(\a^m*_Q z)}_\infty$) holds for at most one positive value of $m$ or   there exist an $(x,y)$-intersection point $R$ such that $\phi_P(X) < \phi_R(X).$
	\end{enumerate}
	\end{proposition}
\begin{proof}
Suppose $\conju{(\a^m*_P y)}_0=\conju{(\a^m*_Q z)}_\infty$ for two distinct values of  $m$. By Proposition \ref{prop:values of m}(2),  $\phi_P(Y)+\phi_Q(Y)=\pi,$ for all $Y \in \T.$ By Lemma \ref{lem:decreasing} this is not possible.  Thus, (1) is proved.

If $\conju{(\a^m*_P y)}_0=\conju{(\a^m*_Q z)}_0,$  for more than two values of $m$ by Proposition~\ref{prop:values of m}(1), we have that  $\ell_{y}(Y)=\ell_{z}(Y)$ and $\phi_P(Y)=\phi_Q(Y)$    for all $Y\in\T$.

Fix  $X\in\T$, any $m$ and consider a geodesic lift $A$ of the geodesic in the free homotopy class of $\conju{(\a^m*_P y)}_0=\conju{(\a^m*_Q z)}_0$. 
    
The result follows them by Lemma~\ref{lem:smaller angle}.

The proof for the case $\conju{(\a^m*_P y)}_\infty=\conju{(\a^m*_Q z)}_\infty$ is similar. Hence, result is proved.
\end{proof}

\begin{lemma}\label{lem:nonca21} 
Let  $\conju{\a}, \conju{\b_1}, \dots, \conju{\b_k}$ be pairwise distinct free homotopy classes of  closed curves such that $\conju{\a}$ contains a  simple representative. Let $\conju{\b}=\sum_{i=1}^k c_i\conju{\b_i}$  where the coefficients $c_1, c_2,\dots, c_k$ are in the ring $\K$. 
Then either  $i(\a,\b_i)= 0$ for all $i\in\{1,\ldots,k\}$ or  there exists a positive integer $m_0$ such that $[\conju{\a}^m,\conju{\b}]\ne 0$ for all $m \ge m_0$ .  
\end{lemma}	
\begin{proof}
From the definition of the  bracket, for any $m\in \N$, 
\begin{align*}
   [\conju{\a}^m,\conju{\b}] &=\sum_{i=1}^kc_i[\conju{\a}^m,\conju{\b}_i]\\
   & = m\sum_{i=1}^kc_i\sum\limits_{P \in \a \cap y_i } \conju{(\a^m*_P y_i)}_0-\conju{(\a^m*_P {\b_i})}_\infty.
 \end{align*}
Fix a metric $X$ and assume that $\a$, $\b_1 \dots, \b_k$ are $X$-geodesics.
 For each $m$,  the sum $\sum_{i=1}^kc_i\sum\limits_{P \in \a \cap y_i } \conju{(\a^m*_P y_i)}_0-\conju{(\a^m*_P {\b_i})}_\infty$ has $I=2(i(\conju{\a}, \conju{\b_1})+i(\conju{\a}, \conju{\b_2})+\cdots+i(\conju{\a}, \conju{\b_k}))$ terms, before performing possible cancellations.  
 
 Choose a metric $X$ in $\T$. Consider $P$, one of the intersection points of the $X$-geodesic in $\a$ and $\bigcup\limits_{1 \le i \le k} y_i$ and choose one of the terms of the bracket associated with $P$, $0$ or $\infty$. For simplicity, assume that the chosen term is $(\conju{\a *_P \b_1})_0$. 
 
 If for all $1 \le m \le I+1$, we have that  $[\conju{\a}^m,\conju{\b}]=0$ then there exist  $m_1, m_2 \in \{1,2,\dots, I+1\}$, and an $X$-geodesic $\b_i$, such that one of the following holds: 
 
 \begin{enumerate}
    \item[(a).]  $(\conju{\a^m *_P \b_1})_0=(\conju{\a^m *_Q \b_i})_0$ for $m=m_1$ and $m=m_2$.
  \item[(b).]  $(\conju{\a^m *_P \b_1})_0=(\conju{\a^m *_Q \b_i})_\infty$ for $m=m_1$ and $m=m_2$.
 \end{enumerate}

 By Proposition~\ref{prop:noncapower}(1), (a) is not possible. Hence (b) holds and by Proposition~\ref{prop:noncapower}(2), there exists a point $R$ in $\a \cap \b_1$, such that  the   angle $\phi$ at $P$ is strictly smaller than the   angle $\phi$ at $R$  (in symbols, $\phi_P  < \phi_R$). This is not possible because $P$ was chosen arbitrarily, and so the proof is complete.
\end{proof}

We will make use of the following well known result. (This result is usually stated for finite type surfaces but it can be generalized to all Riemann surfaces using the fact that every closed geodesic is included in a compact subsurface). 
\begin{lemma}\label{lem:classical}
If $\Sigma$ is an orientable surface 
and  $y$ is a closed curve on $\Sigma$ such that $i(x,y)=0$ for every simple closed curve $\a$. Then $y$ is either homotopically trivial or homotopic to a boundary curve or homotopic to a puncture. 	
\end{lemma}

\subsection{Proof of the Center Theorem}\label{subsec:center}

Suppose that $\conju{y}=\sum_{i=1}^k c_i\conju{y_i}$ belongs to the center,  where the free homotopy classes of $\conju{y}_1, \conju{y}_2,\dots,\conju{y}_k$ are pairwise distinct and $c_i\in\K$ for $i\in\{1,2,\ldots,k\}$. Let $\a$ be any simple closed curve. By definition of center, $[\conju{x}^n,\conju{y}]=0$ for all positive integers $n$.  Therefore Lemma \ref{lem:nonca21} implies that $i(\a,\b_i)=0$ for all $i\in\{1,2,\ldots k\}$. Hence Lemma \ref{lem:classical} implies that each $\conju{\b}_i$ is either homotopically trivial or homotopic to a boundary curve or homotopic to a puncture, which completes the proof.

 
 \section{Universal enveloping algebra and symmetric algebra of TWG-Lie algebra}\label{sec:uea}
  
  Let $\U(\K\conju{\pi})$ be the universal enveloping algebra and $\s(\K\conju{\pi})$ be the symmetric algebra of $K(\conju{\pi})$. 
  For definition and basic properties of these objects see \cite{abe2004hopf}, \cite{hoste1990homotopy}.  
  
  $\U(\K\conju{\pi})$  has a natural Poisson algebra structure with the commutator being the Lie bracket. We extend the Lie bracket of $\K\conju{\pi}$ to $\s(\K(\conju{\pi}))$ using Leibniz rule. This makes $\K(\conju{\pi})$ a Poisson algebra. 
  
  The Poisson center  of a Poisson algebra $(A,\{,\})$ is the subalgebra consists of elements $y\in A$ such that $\{x,y\}=0$ for all $x\in A$. In this section we discuss the Poisson center of the Poisson algebras  $\U(\K\conju{\pi})$ and $\s(\K\conju{\pi})$.  
  
  There are canonical maps from $\K\conju{\pi}$ to $\U(\K\conju{\pi})$ and $\s(\K\conju{\pi})$. To simplify notation, we denote an element and its image under these maps by the same notation. We also denote the product of two elements  in  $\U(\K\conju{\pi})$ and $\s(\K\conju{\pi})$ simply by juxtaposing the elements.

  Let us recall the Poincare-Birkhoff-Witt theorem for $\K\conju{\pi}$(see \cite{abe2004hopf}, \cite[Theorem 3.2]{hoste1990homotopy}).
  
  \begin{theorem}\label{thm:pbw}
  	Let $\leq $ be a fixed total order on $\conju{\pi}$. Consider the set $$S=\{\conju{x}_1\conju{x}_2\cdots \conju{x}_n: n>0, \conju{x}_i\in\conju{\pi}, i\in\{1,2,\ldots ,n\},  \conju{x}_1\leq \conju{x}_2\leq \cdots  \leq \conju{x}_n\}.$$ 
  	Both  $\U(\K\conju{\pi})$  and $\s(\K\conju{\pi})$ are freely generated by $S$ as a $\K$ module. Moreover the natural maps from $\conju{\pi}$ into  $\U(\K\conju{\pi})$  and $\s(\K\conju{\pi})$ are injective Lie algebra homomorphisms. 
  \end{theorem}
  \begin{theorem}\label{thm:univ}
  	The Poisson center of the Poisson algebras  $\U(\K\conju{\pi})$  and $\s(\K\conju{\pi})$ are generated by scalars $\K$, the free homotopy class of constant curve and the curves homotopic to boundary and punctures.
  \end{theorem}
  
  \begin{proof}
  	We compute the center $\U(\K\conju{\pi})$. For $\s(\K\conju{\pi})$, the proof is exactly the same. Let $Z$ be an element of the center of $\U(\K\conju{\pi})$. Then by Theorem \ref{thm:pbw}, $$Z=\sum_{i=1}^n C_i\; \conju{x}_{i_1}\conju{x}_{i_2}\cdots \conju{x}_{i_k}$$ where $\conju{x}_{i_1}\leq \conju{x}_{i_2}\leq \cdots  \leq \conju{x}_{i_k}$ for all $i\in\{1,2,\ldots ,n\}$. For any $\conju{x}\in \K\conju{\pi}$ we have 
  	\begin{align*}
  	0=[\conju{x}^n,Z] &=\sum_{i=1}^nC_i\;[x^n,\conju{x}_{i_1}\conju{x}_{i_2}\cdots \conju{x}_{i_k}]\\
  	&=\sum_{i=1}^{n}C_i\;
  	\sum_{j=1}^{k}n
  	\Bigg\{
  	\sum_{p_{i_j}\in \conju{x}\cap\conju{x}_{i_j}}
  	\Big(
  	\conju{x}_{i_1}\cdots \conju{x}_{i_j-1}(\conju{x}^n*_{P_{i_j}}\conju{x}_{i_j})_0\conju{x}_{i_j+1}\cdots \conju{x}_{i_k}\\
  	& -\conju{x}_{i_1}\cdots \conju{x}_{i_j-1}(\conju{x}^n*_{P_{i_j}}\conju{x}_{i_j})_\infty\conju{x}_{i_j+1}\cdots \conju{x}_{i_k}
  	\Big)
  	\Bigg\}	
  	\end{align*}
  	Fix a metric $X\in\T$ and choose ${x}$ to be any simple $X$-geodesic. Let $x_{i_j}$ be the $X$-geodesic in the free homotopy class of $\conju{x}_{i_j}$. 
  	
  	Fix any $\conju{x}_{i_j}$ and consider $P\in{x}\cap {x}_{i_j}$ such that $\phi_P(x,x_{i_j})\geq \phi_R(x,x_{i_j})$ for all $R\in x\cap x_{i_j}$. As $[\conju{x}^n,Z]=0$ for all positive integer $n$, from the above expression and Theorem \ref{thm:pbw}, there exists $\conju{x}_{k_l}\neq \conju{x}_{i_j}$ such that for infinitely many positive integer $n$, one of the following is true.
  	\begin{itemize}
  		\item $(\conju{x}^n*_{P}\conju{x}_{i_j})_0=(\conju{x}^n*_{Q}\conju{x}_{k_l})_0$ for some $Q\in\conju{x}_{i_j}\cap\conju{x}_{k_l}$. This is impossible by Proposition \ref{prop:noncapower}.
  		\item $(\conju{x}^n*_{P}\conju{x}_{i_j})_0=(\conju{x}^n*_{Q}\conju{x}_{k_l})_\infty$ for some $Q\in\conju{x}_{i_j}\cap\conju{x}_{k_l}$. This is impossible by Proposition \ref{prop:noncapower}.
  		\item $(\conju{x}^n*_{P}\conju{x}_{i_j})_0=(\conju{x}_{k_l})$. This is impossible by Lemma \ref{lem:cosh}.	
  	\end{itemize}  
  	Therefore each $\conju{x}_{i_j}$ is disjoint from $\conju{x}$.  As $\conju{x}$ is an arbitrary simple closed geodesic, each $\conju{x}_{i_j}$ is disjoint from every simple closed geodesic on the surface. Hence by Lemma \ref{lem:classical} each $\conju{x}_{i_j}$ is either a constant loop or a loop homotopic to a  puncture or a loop homotopic to a boundary component.
  \end{proof}

\section{Examples of the Goldman and TWG brackets}\label{sec:concrete}

Recall that, on a path connected space, the set of  free homotopy classes of closed directed  curves is in one to one correspondence with conjugacy classes of the fundamental group. Thus, the set of conjugacy classes of the fundamental group minus the conjugacy class of the trivial loop is in two-to-one correspondence with 
the set of free homotopy classes of oriented closed curves minus the class of the trivial loop. Recall also that the  fundamental group of a  surface with boundary is in one to one correspondence with the set of  reduced words in a fixed minimal set of  generators. Hence, the set of free homotopy classes of closed directed curves on a surface with boundary is in one to one correspondence with the set of cyclic, reduced words on a fixed minimal set of generators. (The empty word is also a reduced word in this discussion)  

We will use the above correspondences in this section, where we give examples of the Goldman and TWG bracket on certain surfaces with boundary, first by choosing a set of minimal generators of the fundamental group, and second, denoting each class of curves by a cyclic word in these generators and their inverses. (Given the two-to-one correspondence mentioned above, there are two possible cyclic words denoting a class of undirected curves).
 To simplify the notation, the inverse of a generator $x$ will be denoted by $X$. 
Also,   $\conju{w}$  will denote the free homotopy class of a representative of  $w$ after removing the orientation and the basepoint.

\begin{figure}[ht]
  \centering
    \includegraphics[trim =105mm 80mm 100mm 45mm, clip, width=6cm]{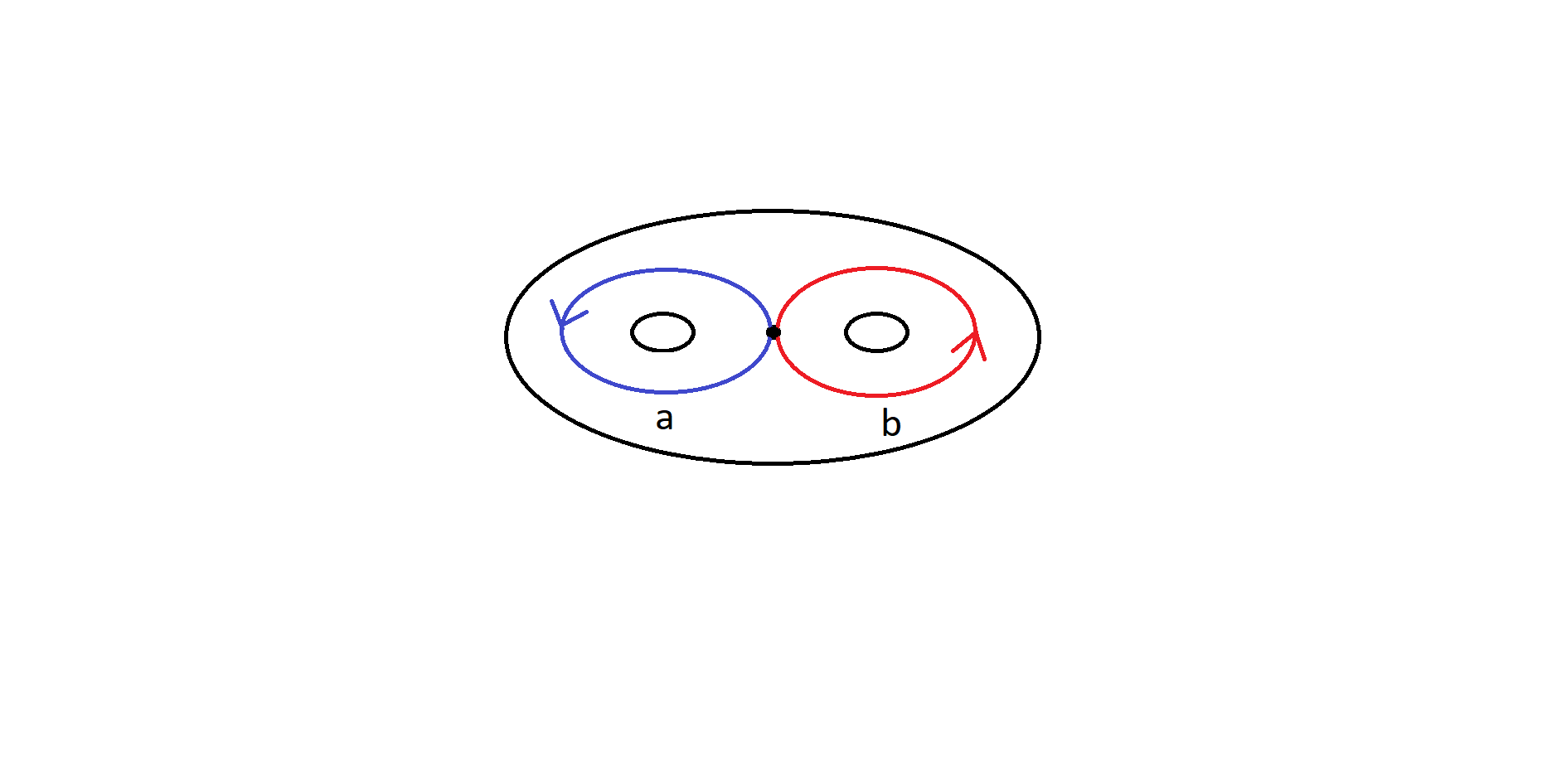}
      \includegraphics[width=4.5cm]{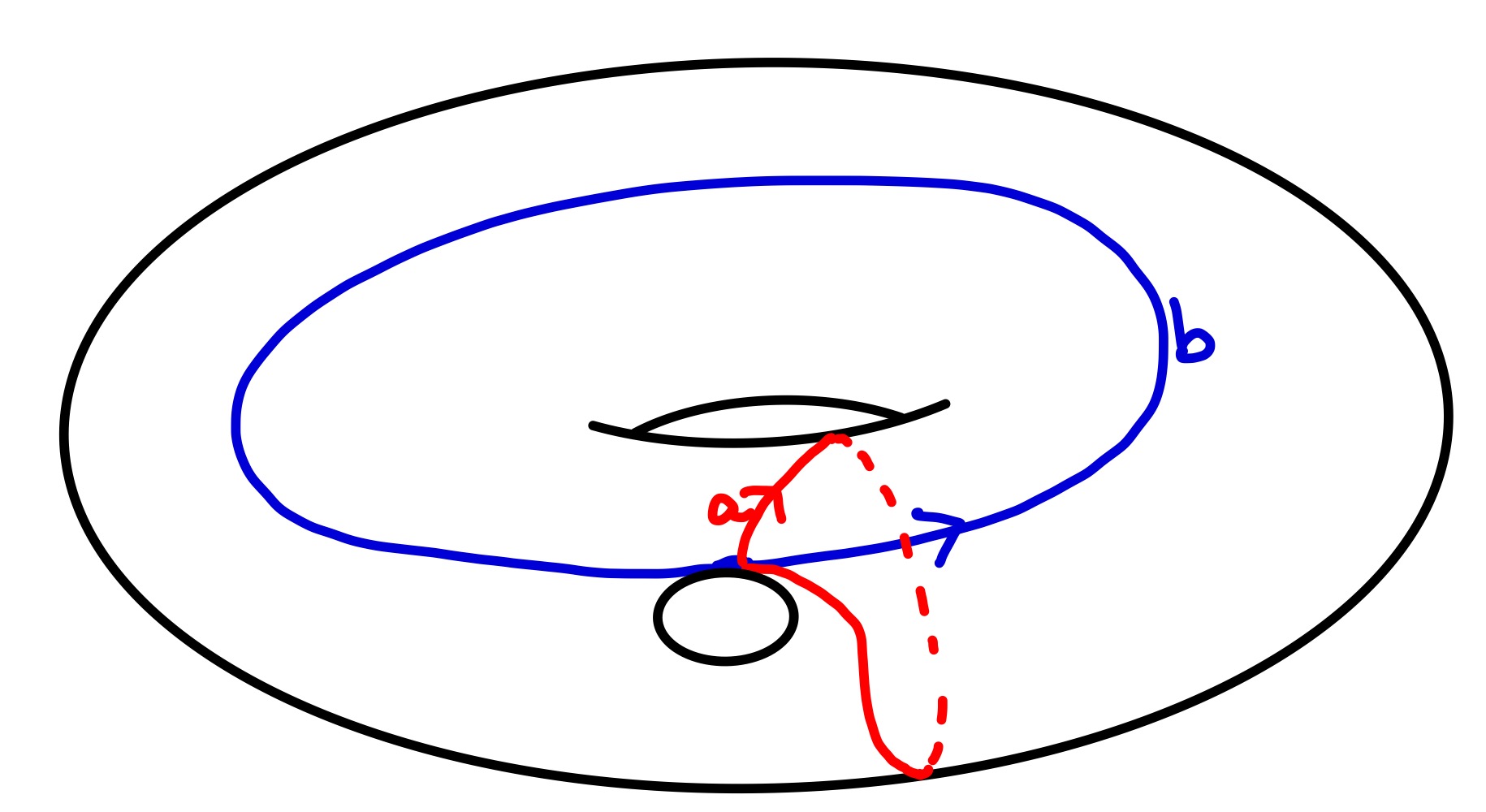}
     \caption{A choice of generators of the pair of pants (left) and the punctured torus (right)
     }
     \label{fig:pop}
\end{figure}

\begin{ex}\label{ex:non-simple} Consider the fundamental group of the triply punctured sphere or pair of pants with generators given by the classes $a$ and $b$ of two curves parallel to two of the boundary components, oriented so that $ab$ goes around the third boundary component (see Figure~\ref{fig:pop}, left).

\begin{figure}[ht]
  \centering

        \includegraphics[width=12cm]{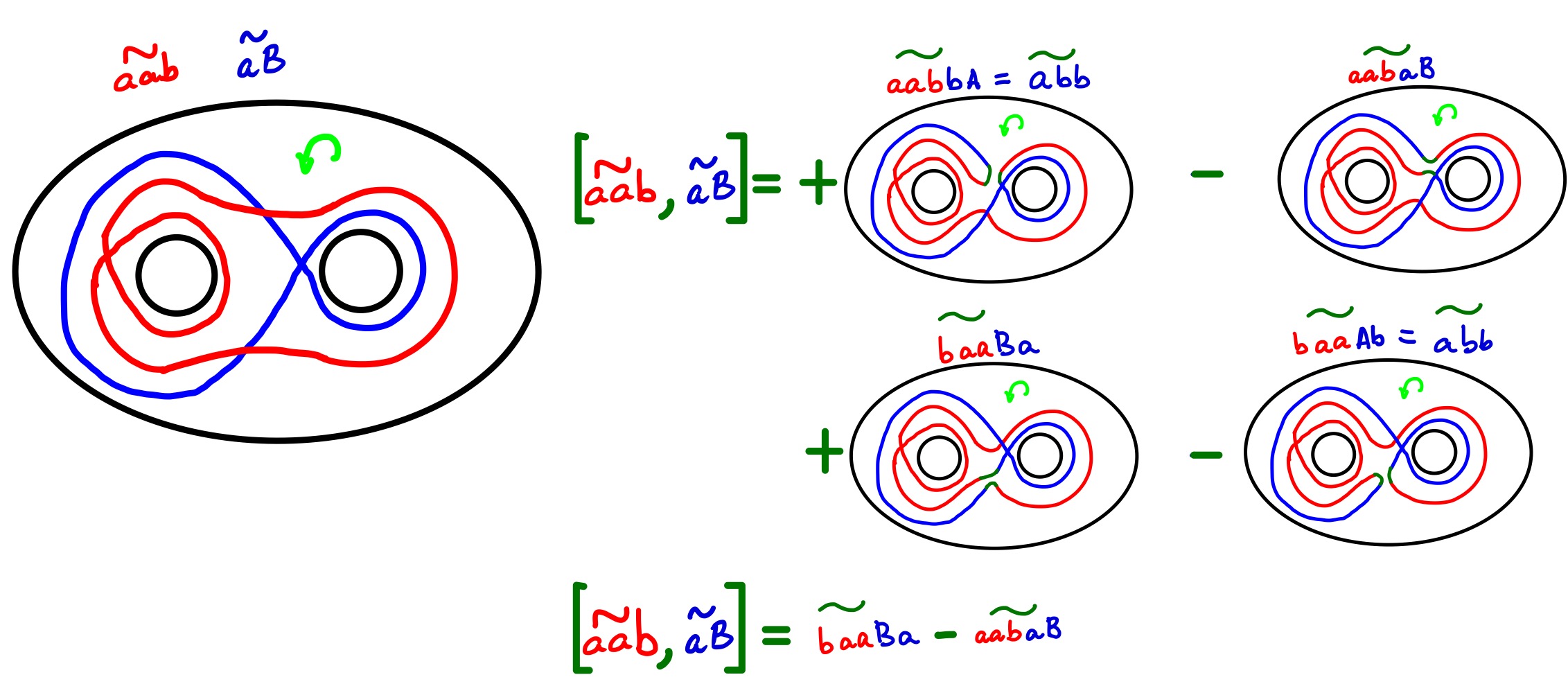}
    
     \caption{Example $[\conju{aab},\conju{aB}]$: Representatives of $\conju{aab}$ and $\conju{aB}$ (left) and computation of the bracket (right)}
     \label{fig:example aab aB}
\end{figure}

\begin{center}
$[\conju{aab},\conju{aB}]=\conju{baaBa}-\conju{Baaba},$\\
\end{center}

\end{ex}

Example~\ref{ex:non-simple} shows the need of the hypothesis of one of the curves being simple in the Intersection Counting Theorem  since $i(\conju{aab},\conju{aB})=2$ but the number of terms of the bracket is less than twice the intersection number, which is  $4$.

\begin{ex}\label{ex:non-simple directed} In the same setting as Example~\ref{ex:non-simple}, it is not hard to check that the Goldman bracket of the conjugacy classes of the directed curves $aaB$ and $aB$  and the two curves have intersection number equal to $2$. 

\end{ex}

\begin{ex}\label{ex:simple} If we consider the punctured torus with fundamental group with standard generators labeled $a$ and $b$  then
 $[\conju{abAb},\conju{aB}]=\conju{aBBB}-\conju{ABaBAb}
 +\conju{AB}-\conju{aBABaB}.$
\begin{figure}[ht]
  \centering
        \includegraphics[width=14cm]{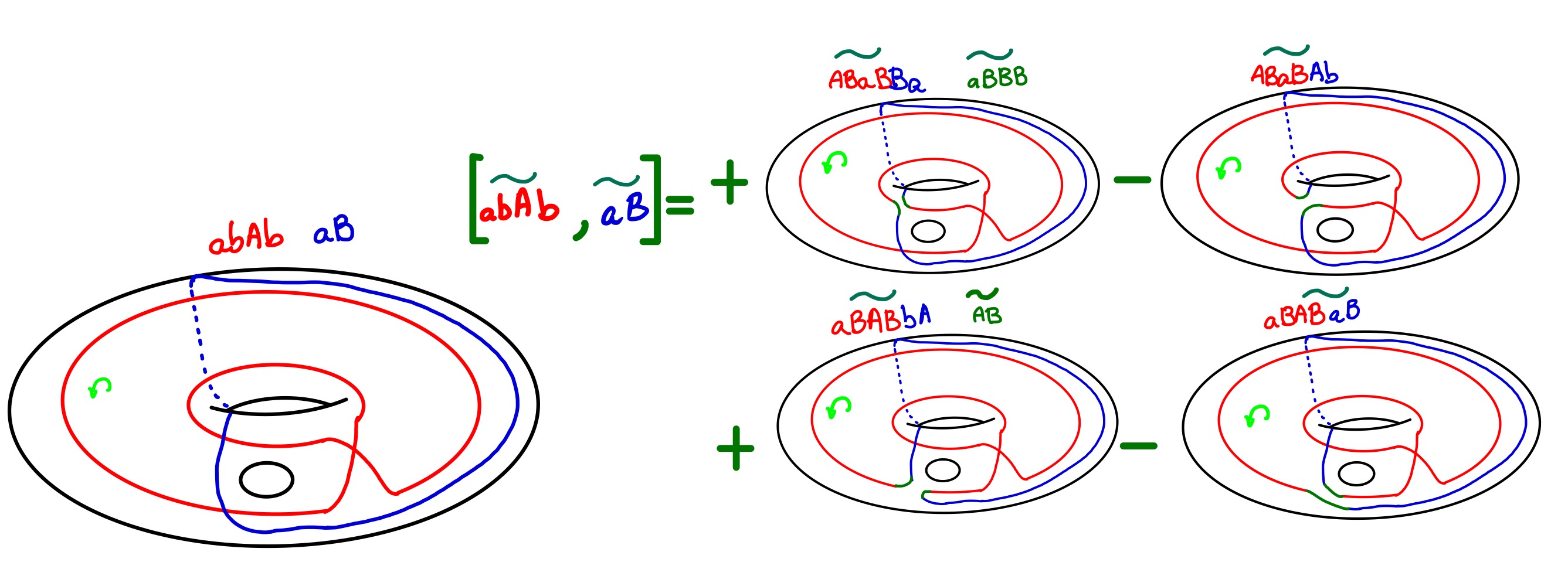}
          \caption{Example $[\conju{abAb},\conju{aB}]$:  Representatives of $\conju{abAb}$ and $\conju{aB}$ (left) and computation of the bracket (right)}
     \label{fig:example torus}
\end{figure} 
\begin{center}
  $[\conju{aab},\conju{aB}]=\conju{baaBa}-\conju{Baaba},$\\
  $[\conju{aaB},\conju{aB}]=\conju{aaBAb}-\conju{aabAB}.$
 \end{center}
\end{ex}

Example~\ref{ex:simple}  illustrates Intersection Counting Theorem.
 In this case,   $i(\conju{aB},\conju{aab})=2$ and the number of terms of the bracket $[ \conju{aab},\conju{ab}]$ is $4$.

\begin{namedtheorem}[Computational] Consider two classes of  curves $\conju{x}$ and $\conju{y}$, If one of the following holds
\begin{itemize}
    \item The word length of  $\conju{x}$ and $\conju{y}$ is less than or equal to $12$ and $x$ and $y$ are in the punctured torus.
    \item The word length of  $\conju{x}$ and $\conju{y}$ is less than or equal to $11$ and $x$ and $y$ are in the pair of pants.
    \item The word length of  $\conju{x}$ and $\conju{y}$ is less than or equal to $8$ and $x$ and $y$ are in $aAbBcC$ the four holed sphere.
    \item The word length of  $\conju{x}$ and $\conju{y}$ is less than or equal to $7$ and $x$ and $y$ are in the punctured genus two surface with  surface word $abABcdCD$.

\end{itemize}
 and the bracket of $\conju{x}$ and $\conju{y}$ is zero then  $\conju{x}$ and $\conju{y}$ have disjoint representatives. In symbols, if  $[\conju{x},\conju{y}]=0$ then $i(\conju{x},\conju{y})=0.$ 

\end{namedtheorem} 

The previous Computational Theorem lead us to the following conjecture.

\begin{conj*} If the TWG-Lie bracket of two classes of undirected curves is zero, then the classes have disjoint representatives.
\end{conj*}

\section{Proof of the Jacobi identity for the TWG-bracket}\label{sec:jacobi}

A pair of points $(P,Q), P \in x \cap y, Q \in y\cap z$ determines four terms of the double bracket  $[[\conju{x},\conju{y}],\conju{z}]$, each of these four types represented in the upper row of Figure~\ref{fig:jacobi}.
Repeating this reasoning for the other two brackets, $[[\conju{y},\conju{z}],\conju{x}]$ and $[[\conju{z},\conju{x}],\conju{y}]$, it is not hard to check the Jacobi identity.
$$[[\conju{x},\conju{y}],\conju{z}]+[[\conju{y},\conju{z}],\conju{x}] +[[\conju{z},\conju{x}],\conju{y}]=0.$$

\begin{figure}[ht]
  \centering
    \includegraphics[width=\textwidth]{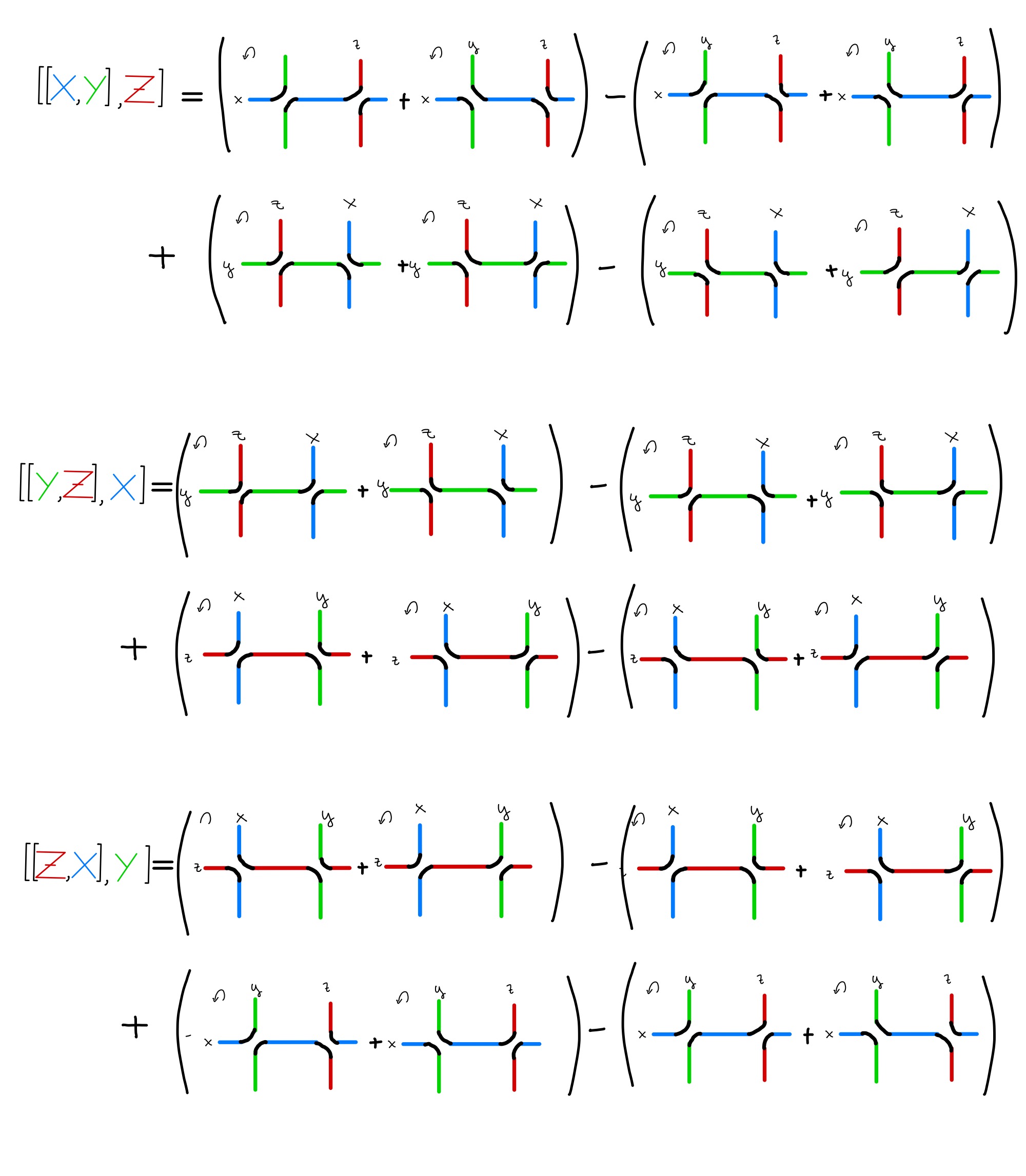}
     \caption{Representation of the double brackets $[[\conju{x},\conju{y}],\conju{z}]$, $[[\conju{y},\conju{z}],\conju{x}]$, and  $[[\conju{z},\conju{x}],\conju{y}]$}\label{fig:jacobi}
\end{figure}
\section{Goldman's definition of the TWG-bracket}\label{sec:Goldman}

In this section we review the definition of the Goldman Lie algebra on directed curves,  Goldman's definition of the TWG-Lie algebra of undirected curves and prove that there is an isomorphism between Goldman's  and the definition of the TWG-Lie algebra we gave in the introduction.

Denote by $\pi$ the set of  free homotopy classes of directed closed curves on $\Sigma$,  by $\la \ox\ra $ the free homotopy class of a   directed closed curve $\ox$,   and by $\K \pi$ the free module spanned by $\pi$. 

The \emph{Goldman Lie bracket on $\K\pi$} is the linear extension to $\K\pi$ of the bracket of two free homotopy classes $\la \ox\ra $ and $\la \oy\ra $ defined by
$$[\la \ox\ra,\la \oy\ra]=\sum_{P\in \ox\cap \oy}\e_P\la \ox *_P\oy\ra.$$ Here, the representatives $\ox$ and $\oy$ are chosen so that they intersect  transversely in a set of double points  $\ox \cap \oy$, 
$\e_P$ denotes the sign of the intersection between $\ox$ and $\oy$ at an intersection point $P$, and $\ox *_P\oy$ denotes the loop product of ${\ox}$ and ${\oy}$ at $P$.

Goldman \cite{goldman_invariant_1986} proved that this bracket is well defined, skew-symmetric and satisfies the Jacobi identity on $\K\pi$.  In other words, $\K\pi$  is a Lie algebra.

There is a natural involution  $\map{\iota}{\pi}{\pi}$ defined by $\iota(\la \ox\ra)=\la \bar{\ox}\ra$ where $\bar{\ox}$ denotes the curve $\ox$ with opposite orientation. By extending $\iota$ linearly to $\K \pi$ we obtain a $\K$ -linear involution   $\smap{\iota}{\K\pi}$. The invariant subspace of $\iota$, denoted by $S$   is   a Lie subalgebra of  $\K\pi$ \cite[Subsection 5.12]{goldman_invariant_1986}. 

Now we prove that the  subalgebra $S$ is isomorphic to the TWG-Lie algebra $\K\conju{\pi}$.  

First observe that $S$ is generated by the elements of the form $\la\ox\ra+\la \bar{\ox} \ra$.  A straightforward computation in $\K\conju{\pi}$ shows that
\begin{equation}\label{eq:computation}
[\la\ox\ra+\la \bar{\ox} \ra,\la\oy\ra+\la \bar{\oy} \ra]=[\la\ox\ra,\la\oy\ra]+\overline{[\la\ox\ra,\la\oy\ra]}+[\la\ox\ra,\la\bar{\oy}\ra]+\overline{[\la\ox\ra,\la\bar{\oy}\ra]}   
\end{equation}
where the ``change direction''  operator $\overline{~\cdot~}$ is extended to $\K\conju{\pi}$ by linearity.

Next define a map from $S$ to  $\K\conju{\pi}$, by sending each element of $S$  of the form $\la \ox\ra+\la\bar{\ox}\ra$ (that is, each element of a the geometric basis of $S$) to the undirected free homotopy class $u_{\la \ox \ra}$ defined by ``forgetting'' the direction of $\ox,$ and considering the free homotopy class.  Extend $u$ to $S$ by linearity.

\begin{figure}[ht]
  \centering
    \includegraphics[trim = 10mm 55mm 40mm 35mm, clip, width=12cm]{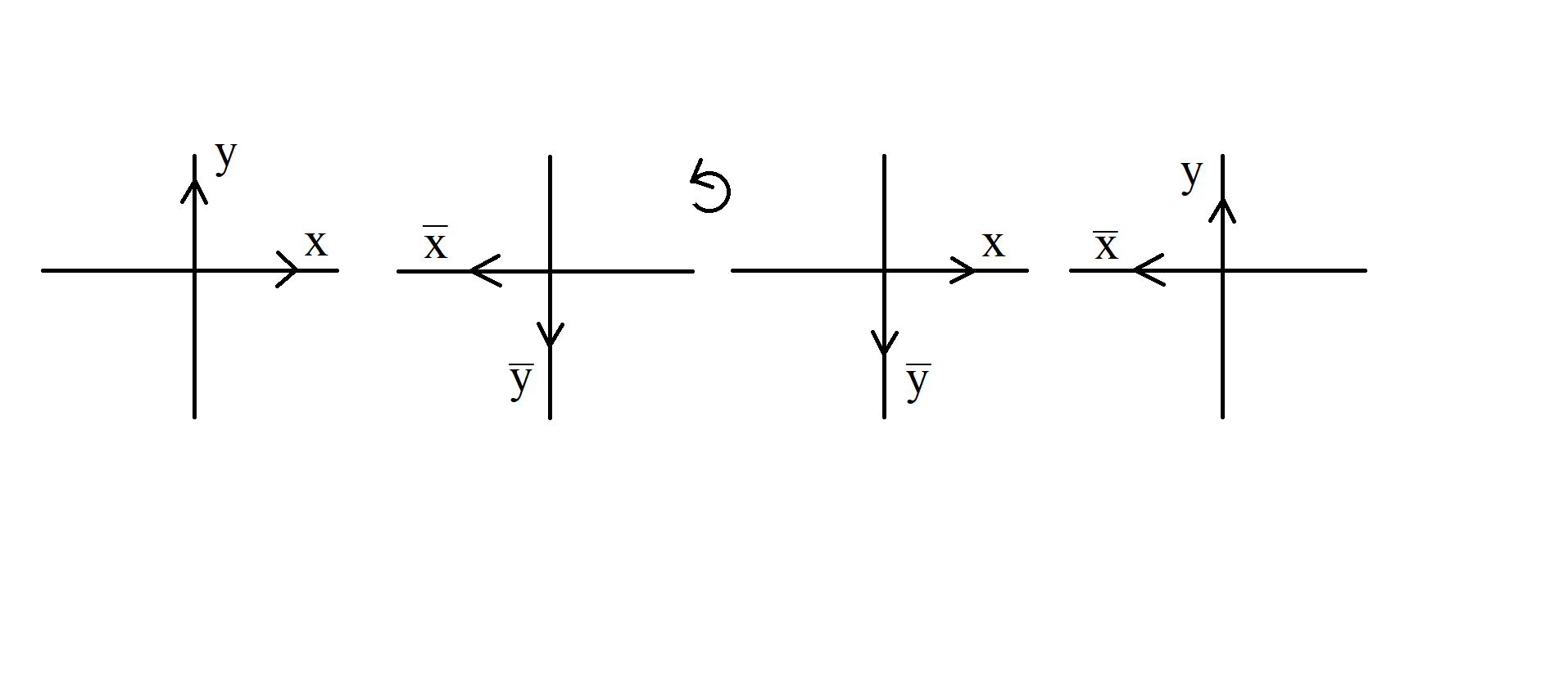}
     \caption{Sign of the elements.}\label{fig:orientation}
\end{figure}

Observe that (Figure \ref{fig:orientation})
$$\e_P(\ox,\oy)=\e_P(\bar{\ox},\bar{\oy})=-\e_P(\ox,\bar{\oy})=-\e_P(\bar{\ox},\oy).$$
Also, if $\e_P(\ox,\oy)=1$ then 
$$u_{\la \ox *_P \oy\ra} = \conju{ ( \ox *_P \oy)}_0\,\, \text{and} \,\,u_{\la \ox *_P \bar{\oy}\ra} = \conju{ ( \ox*_P \oy)}_\infty.$$
 and if $\e_P(\ox,\oy)=-1$ then  
$$u_{\la \ox *_P \oy\ra} = \conju{ ( \ox*_P \oy)}_\infty\,\, \text{and} \,\,u_{\la \ox *_P \bar{\oy}\ra} = \conju{ ( \ox*_P \oy)}_0.$$

Therefore computing the bracket in $\K\conju{\pi}$, we observe 
$$
[u_{\la \ox\ra},u_{\la \oy\ra}]=
\sum_{P \in \ox \cap \oy, \e_P=1}
\left(u_{\la\ox *_P \oy \ra}- u_{\la\ox *_P \bar{\oy}\ra}\right)  - \sum_{P \in \ox \cap \oy, \e_P-1}\left(u_{\la\ox *_P \bar{\oy}\ra }- u_{\la\ox *_P \oy\ra}\right) 
$$
$$
=\sum_{P \in \ox \cap \oy} \e_P \left(u_{\la\ox *_P \oy \ra}+ u_{\la \ox *_P \bar{\oy}\ra}\right)= u_{[\la\ox\ra,\la\oy\ra]}+u_{[\la\ox\ra,\la\bar{\oy}\ra]}
$$
On the other hand, by Equation~(\ref{eq:computation}) by applying $u$ to the bracket $[\la\ox\ra+\la \bar{\ox} \ra,\la\oy\ra+\la \bar{\oy} \ra]$ we obtain 
$u_{[\la\ox\ra,\la\oy\ra]}+u_{[\la\ox\ra,\la\bar{\oy}\ra]}.$ This shows that $u$ is a Lie algebra isomorphism, as desired.

\section{Proof of the Canonical Decomposition Theorem}\label{sec:candec}

We now prove the Canonical Decomposition Theorem. 
Consider two undirected classes of curves $\alpha$ and $\beta$, not null-homotopic and not parallel to a boundary component in the surface $\Sigma$. Fix a hyperbolic metric $X$ on a surface, which is homotopy equivalent to $\Sigma$, so that all boundary components are punctures. 
Let $x$ and $y$ be $X$-geodesic representatives of $\alpha$ and $\beta$. By Theorem~\ref{theo:beardon},  both $(x*_P y)_0$ and $(x*_P y)_\infty$  are hyperbolic and have positive length. Thus, they are not parallel to boundary components and not null-homotopic.

The proof in the case of the Goldman Lie algebra follows similarly.

\bibliographystyle{plain}

\bibliography{gla.bib}
\end{document}